%% file: papersubgrad_methods_V3.tex
\documentclass[onefignum,onetabnum]{siamonline220329}

\input{shared}
\ifpdf
\hypersetup{
  pdftitle={Subgradient Methods for Langevin Sampling from Non-smooth Potentials},
  pdfauthor={}
}
\fi

\begin{document}

\pgfkeys{/prAtEnd/global custom defaults/.style={end,
text link={For a proof see \Cref{appendix_proofs}.}
}}

\maketitle

\begin{abstract}
This paper is concerned with sampling from probability distributions $\pi$ on $\R^d$ admitting a density of the form $\pi(x) \propto e^{-U(x)}$, where $U(x)=F(x)+G(Kx)$ with $K$ being a linear operator and $G$ being non-differentiable. Two different methods are proposed, both employing a subgradient step with respect to $G\circ K$, but, depending on the regularity of $F$, either an explicit or an implicit gradient step with respect to $F$. For both methods, non-asymptotic convergence proofs are provided, with improved convergence results for more regular $F$. Further, numerical experiments are conducted for simple 2D examples, illustrating the convergence rates, and for examples of Bayesian imaging, showing the practical feasibility of the proposed methods for high dimensional data.
\end{abstract}

\begin{keywords}
Markov chain Monte Carlo methods, Unadjasted Langevin algorithm, non-smooth sampling,\\Bayesian inference, Bayesian imaging, inverse problems.
\end{keywords}

\begin{MSCcodes}
65C40 $\cdot$ 65C05 $\cdot$ 68U10 $\cdot$ 65C60
\end{MSCcodes}

\section{Introduction}
In this work we propose algorithms for sampling from a probability distribution $\pi$ on $\R^d$ admitting a density - also denoted $\pi$ - of the form
\begin{equation}\label{eq:intro:distribution}
\pi(x) = \frac{e^{-U(x)}}{\int e^{-U(y)\;\d y}}
\end{equation}
with a \emph{potential} $U:\R^d\rightarrow [0,\infty)$, $U(x)=F(x)+G(Kx)$ where $F:\R^d\rightarrow [0,\infty)$, $G:\R^{d'}\rightarrow [0,\infty)$ are both proper, convex, and lower semi-continuous (lsc), and $K:\R^d\rightarrow \R^{d'}$ is a linear operator. Such exponential - or Gibbs - densities arise frequently in applications and are derived by a minimax entropy principle for image distributions in \cite{zhu1997minimax}. In particular, we are interested in the non-smooth setting that $G$ or both, $F$ and $G$, are non-differentiable. More specifically, we assume the following regularity conditions.
\begin{ass}\label{ass:general} The functions $F$ and $G$ satisfy the following regularity assumptions:
\begin{enumerate}[(1)]
\item $F$ is convex and 
\begin{enumerate}[(a)]
\item either Lipschitz continuous with constant $L_F$\label{ass:Fa}
\item or differentiable with Lipschitz continuous gradient with constant $L_{\nabla F}$.\label{ass:Fb}
\end{enumerate}
\label{ass:F}
\item $G$ is convex and Lipschitz continuous with constant $L_G$.\label{ass:G}
\end{enumerate}
\end{ass}
The task of sampling from such distributions frequently arises, e.g., in Bayesian inference in the field of mathematical imaging, see for instance \cite{pereyra2016proximal,durmus2022proximal} for more background information, \cite{zach2022stable} and \cite{zach2022computed} for applications in magnetic resonance imaging and computed tomography, respectively, and \cite{narnhofer2022posterior} for error quantification based on sampling techniques. In this context, one is interested in inferring an unknown high-dimensional random variable $X\in\R^d$ based on observing a random variable $Y\sim \p(y|x)$. The distribution $\p(y|x)$ and the so-called \emph{prior} distribution $\p(x)$ are subject to modeling decisions and assumed to be known for our purposes. Using Bayes' Theorem, the logarithm of the \emph{posterior} distribution of $X|Y=y$ is easily seen to be
\begin{equation}
\begin{aligned}
\log\p(x|y) = &\log\p(y|x) + \log \p(x) - \log \p(y).
\end{aligned}
\end{equation}
Given an observation of $Y$, sampling from the posterior then allows not only to estimate $X$ by computing, e.g., its empirical expected value, but also to gauge the uncertainty of its estimate \cite{narnhofer2022posterior}. The function $x\mapsto\log\p(y|x) + \log \p(x) - \log \p(y)$ here corresponds to $-F(x)-G(Kx)$ in our general model. The specific structure of composing a non-differentiable functional $G$ with a linear operator $K$ allows for using a wide range of models for the log-prior such as $-\log \p(x)\sim \|K x\|_1$ where $K$ can be, e.g., any transform promoting sparsity in a certain basis or frame, or a finite differences operator to obtain the famous Rudin-Osher-Fatemi $\TV-L^2$ model \cite{rudin1992tv_mh} or higher-order variants \cite{holler20ip_review_mh} from inverse problems and imaging\footnote{To be precise, $\p(x)\propto \exp(-\|K x\|_1)$ frequently does not induce a well-defined probability distribution as $\int \p(x)\;\d x=\infty$ if $K$ admits a non-trivial kernel. However, in conjunction with a suitable likelihood function, e.g., $-\log \p(y|x) = F(x) = \|x-y\|^2_2$ for Gaussian denoising, the posterior distribution $\p(x|y) \propto \exp (-F(x) -\|K x\|_1)$ is, in fact, well defined.}. %

In this work, we propose two different methods for sampling from $\pi$. Both employ a subgradient step with respect to $G\circ K$, but depending on the regularity of $F$, either an explicit or an implicit gradient step with respect to $F$. We provide proofs of convergence for all methods, with improved results under stronger regularity assumptions on $F$.

Typically, sampling from \eqref{eq:intro:distribution} is performed using Markov chain Monte Carlo methods. The strategy is to generate a Markov Chain $(X_k)_k$ for which it is shown that the distribution of $X_k$ approximates $\pi$ as $k\rightarrow\infty$. A popular class of such methods are discretizations of the over-damped Langevin diffusion \cite{rossky1978brownian,parisi1981correlation}
\begin{equation}\label{eq:intro:Langevin_SDE}
\d Y_t = -\nabla U(Y_t)\;\d t + \sqrt{2}\;\d B_t
\end{equation}
with $B_t$ denoting Brownian motion. Interestingly, as $U=-\log(\pi)$ occurs solely in differentiated form in \eqref{eq:intro:Langevin_SDE}, such methods rely on knowledge of $\pi(x)$ only up to a scalar multiple, that is, knowledge of the normalization constant $\int e^{-U(y)}\;\d y$ is not necessary. Indeed, under sufficient regularity conditions on $U$, respectively $\nabla U$, \cref{eq:intro:Langevin_SDE} admits $\pi$ as a stationary distribution and the distribution of $Y_t$ converges to $\pi$ in total variation norm \cite[Theorem 2.1]{roberts1996exponential} or in Wasserstein distance \cite{bolley2012convergence}. A discrete time Markov chain (MC) approximating \eqref{eq:intro:Langevin_SDE} is obtained, e.g., via the Euler-Maruyama discretization, referred to as the Unadjusted Langevin Algorithm (ULA)
\begin{equation}\label{eq:intro:Langevin_discr}
X_{k+1} =X_k -\gamma_{k+1} \nabla U(X_{k}) + \sqrt{2\gamma_{k+1}}\; B_{k+1}
\end{equation}
where $\gamma_k>0$ is the step size and $(B_k)_k$ an i.i.d. sequence of $d$-dimensional standard Gaussian random variables. For $\gamma_k=\gamma$ for all $k$, this MC admits a stationary distribution $\pi_\gamma$ which is, in general, not equal to $\pi$, but approximates $\pi$ in the sense that $\pi_\gamma\rightarrow \pi$ as $\gamma\rightarrow 0$ \cite{durmus2019analysis}. For $\gamma_k\rightarrow 0$ with $\sum_k \gamma_k =\infty$, however, a direct approximation of $\pi$ without any bias is feasible \cite{lamberton2003recursive}. Moreover, non-asymptotic convergence bounds have been shown in the total variation norm \cite{durmus2017nonasymptotic,dalalyan2017theoretical} as well as in Wasserstein distance  \cite{durmus2019high}.

\section{Related Works}
The above mentioned works rely on differentiability of the potential $U$, which is undesirable as many relevant applications naturally lead to non-differentiable potentials, e.g., regression with $L^1$ loss or Lasso prior, or non-smooth regularization functionals in Bayesian imaging or inverse problems. Remedies for this issue have been proposed in the literature. For instance, in \cite{durmus2022proximal,pereyra2016proximal} the authors propose to replace non-differentiable parts of the potential with its Moreau envelope leading to a continuously differentiable approximation of the potential as the Moreau envelope $h^\lambda$ with parameter $\lambda>0$ of a (possibly non-differentiable) functional $h$ admits the Lipschitz continuous gradient $\nabla h^\lambda(x) = \lambda^{-1}(x-{\prox}_{\lambda h}(x))$. This effectively leads to a proximal step with respect to non-differentiable parts of the potential $U$. 
In \cite{durmus2022proximal} it is shown that the obtained smooth approximation of the target density can be made arbitrarily accurate in total variation (TV) by decreasing the Moreau-Yoshida parameter. The resulting method is coined MYULA (Moreau Yoshida Unadjusted Langevin Algorithm). In the method from \cite{pereyra2016proximal}, referred to as P-MALA (Proximal Metropolis Adjusted Langevin Algorithm), on the other hand, the discrepancy between the approximation and the true target density is accounted for by introducing a Metropolis-Hastings correction within the method. In \cite{cai2022proximal}, P-MALA is used for Bayesian model selection using a nested sampling procedure. In view potentials $U$ of the form $U(x) = F(x) + G(Kx)$, a disadvantage of MYULA and P-MALA is that the proximal mapping of $G\circ K$ will in general be accessible only via an iterative algorithm, rendering these approaches computationally expensive. In \cite{brosse2017sampling} the approach of \cite{durmus2022proximal} is investigated for the particular case of $G\circ K$ being an indicator function on a compact, convex set. In \cite{luu2021sampling} the authors extend the P-MALA approach to the case of non-convex potentials. It is shown that the smooth approximation of the density can be made arbitrarily accurate in TV. The authors also prove that the resulting continuous time Langevin diffusion admits a unique solution and a stationary distribution. Despite the weakened assumptions, convergence of the discretization of the Langevin diffusion is shown in expectation. In \cite{laumont2022bayesian}, similarly to MYULA, the authors consider smoothing the target density. This time, however, the smoothing is achieved by convolving the non-smooth part with a Gaussian kernel. The gradient of the smoothed potential can be computed using Tweedie's formula which leads to a Plug and Play approach \cite{laumont2022bayesian}. In \cite{durmus2019analysis} the authors propose a subgradient as well as a proximal gradient method for direct sampling from non-smooth potentials without prior smoothing. While the former would allow for sampling from \eqref{eq:intro:distribution} in principle, it yields only suboptimal convergence behavior whenever $F$ is differentiable. Moreover, the subgradient method relies on Lipschitz continuity of $F$, which prevents its applications, for example, to models with $L^2$ loss. The proximal gradient method from \cite{durmus2019analysis}, on the other hand, guarantees better convergence behavior, but again suffers from the fact that the proximal mapping of $G\circ K$ is, in general, not explicit. Another approach has been proposed in \cite{ehrhardt2023proximal} where the authors show that convergence can be maintained despite computing the proximal map only approximately, thus allowing also for its iterative computation within the algorithm. In \cite{bianchi2019passty} the authors propose a method that employs only the proximal mappings of $F$ and $G$, but relies on invertibility of $K$, which is not given in many cases of interest, such as $K$ being a finite differences operator. In \cite{salim2019stochastic} the authors propose a stochastic proximal-gradient sampling method, where they assume the representations $F(x) = \E_\zeta[f(x,\zeta)]$, and $G(x) = \sum_{i=1}^n G_i(x)$, $G_i(x) =\E_\zeta[g_i(x,\zeta)]$. While $f$ is assumed to admit a Lipschitz continuous gradient with respect to $x$, the functions $g_i$ can be non-smooth in $x$, however, their proximal mappings need to be accessible. In \cite{salim2020primal} the authors provide a duality result for the optimization problem in the space of probability measures corresponding to the proximal sampling method from \cite{durmus2019analysis}.

In \cite{liang2022proximal,chen2022improved,lee2021structured} the authors consider the regularized density $\tilde{\pi}(x,y)\propto\exp(-U(x) - \mu\|x-y\|^2)$ from which they sample using a Gibbs algorithm, where they sample alternately, once using a so-called \emph{restricted Gaussian oracle} to draw samples from $\tilde{\pi}(x|y)$ and once from the simple Gaussian $\tilde{\pi}(y|x)$. This framework is cast as Alternating Sampling Framework (ASF). Note that, while $\tilde{\pi}(x|y)$ corresponds to a regularized version of $\pi(x)$, the marginal $\tilde{\pi}(x) = \int \tilde{\pi}(x,y)\;\d y \propto \exp(-U(x)) \int \exp(- \mu\|x-y\|^2)\;\d y \propto \pi(x)$ is in fact equal to the target density since latter integral is independent of $x$. The restricted Gaussian oracle is implemented via a rejection sampling subroutine. In \cite{liang2022proximal} the authors focus on non-differentiable but Lipschitz continuous potentials. In \cite{lee2021structured} these techniques are applied to potentials of the form $U(x)=F(x)+G(x)$ with $F$ strongly convex and a Lipschitz continuous gradient $\nabla F$ but possibly non-smooth $G$. In \cite{chen2022improved} the convergence results from \cite{lee2021structured} are extended to the case of potentials which are not strongly convex. A downside of these alternating sampling methods is the requirement of subroutines for the computation of the restricted Gaussian oracle which might render iterations more expensive than in the case of Langevin based algorithms.

Another related line of works are so-called \emph{asymptotically exact data augmentation models} \cite{vono2021asymptotically,rendell2021global}. There, similarly to above, a joint distribution $\tilde{\pi}_\rho(x,y)$ is introduced, which depends on a parameter $\rho$. Instead of requiring the marginal distribution to be equal to the target, $\tilde{\pi}_\rho(x)=\pi(x)$, however, it is only required that $\tilde{\pi}_\rho(x)\rightarrow\pi(x)$ as $\rho\rightarrow0$ for all $x$ which implies convergence also in the total variation norm by Scheff\'{e}'s Lemma \cite{scheffe1947useful}. In this context, in \cite{vono2019split} the authors consider - in an ADMM fashion - for $\pi(x)\propto \exp(-F(x)-G(x))$ the approximation $\tilde{\pi}_\rho(x,y)=\exp(-F(x)-G(y)-\phi(x,y;\rho))$ with a suitable \emph{distance} function such that $\tilde{\pi}_\rho(x)\rightarrow\pi(x)$. Sampling is then again performed using a Gibbs sampler. This splitting approach could be interesting for the task considered in the present work in particular when also $F$ is non-smooth. Similarly,  in \cite{vono2022efficient}, the authors approximate distributions of the form $\pi(x)\propto \exp(-U(Kx))$ with a linear operator $K$ by $\tilde{\pi}_\rho(x,y)=\exp(-U(y)-\|Kx-y\|^2_2/(2\rho^2))$ such that a decoupling of operator and functional is achieved, for which Gibbs sampling is feasible. A downside of these approaches might be the necessity of a subsampling routine within the Gibbs sampler. As such a subroutine, in \cite{vono2019split} the authors propose to use MYULA \cite{durmus2022proximal} or P-MALA \cite{pereyra2016proximal}, whereas in \cite{vono2022efficient} rejection sampling is utilized.

In the present work we propose two methods for sampling from the potential $\pi(x)$ in \eqref{eq:intro:distribution}, a Proximal-subgradient (Prox-sub) as well as a Gradient-subgradient (Grad-sub) method. The methods employ only the subgradient of $G$, evaluation of $K$ and $K^*$ and the gradient, respectively proximal mapping, of $F$. Inspired by the results and techniques introduced in \cite{durmus2019analysis}, our main algorithm combines the benefits of different methods of \cite{durmus2019analysis} in the sense that it is applicable in the non-smooth case while simultaneously ensuring superior convergence in the smooth case. For a comparison of the complexities of the proposed methods to the state of the art we refer to \Cref{table:overview}, where we show the computational complexity to reach a given accuracy of the sample distribution to the target in the Wasserstein-2 distance, the Kullback-Leibler divergence (KL) and the TV. While in some cases the proposed methods exhibit a greater computational effort, its main advantage is the broad applicability to settings where for other algorithms a subroutine for the computation of proximal mappings is necessary. In summary the contributions of the present work are the following:
\begin{itemize}
\item We propose two Langevin algorithms for sampling from potentials of the form $U(x)=F(x)+G(Kx)$, with $G$ non-differentiable and the proximal mapping of $G \circ K$ not being explicit. To the best of our knowledge, this is the first Langevin-type sampling algorithm (without inner iterations) that is applicable to a potential of the above form, with $F$ not necessarily being Lipschitz continuous. In particular, it is the first algorithm that allows guaranteed sampling from models like the the $\TV-L^2$ model \cite{rudin1992tv_mh} or generalizations to inverse problems with $L^2$ data discrepancy \cite{holler20ip_review_mh}.
\item We provide non-asymptotic convergence results for the proposed algorithms for the case that $F$ is non-differentiable as well as improved results for the case that $F$ satisfies stronger regularity assumptions. The results are non-asymptotic in the sense that a desired accuracy is reached after a finite and explicitly known number of iterations.
\item We provide comprehensive numerical results confirming all proven convergence rates in a 2D setting. The practical feasibility of the methods for high-dimensional data is further demonstrated via examples from Bayesian imaging.
\end{itemize}
\begin{table}[h]
\resizebox{\textwidth}{!}{%
\centering
        \begin{tabular}{cccccccccc}
        \toprule 
        & \multicolumn{9}{c}{Assumptions on $F$}\\
& \multicolumn{2}{c}{Lipschitz continuous} & & \multicolumn{2}{c}{Lipschitz gradient} & & \multicolumn{3}{c}{\makecell{Lipschitz gradient,\\strongly convex}}\\
\cmidrule{2-3}\cmidrule{5-6}\cmidrule{8-10}
 & TV & KL &  & TV & KL & & TV & KL & $W^2_2$\\
\midrule
Prox-sub & $d\epsilon^{-6}$ & $d\epsilon^{-3}$ &  & $d\epsilon^{-4}$ & $d\epsilon^{-2}$ & & $d\epsilon^{-4}$ & $d\epsilon^{-2}$ & $d\frac{\log(\epsilon^{-1})}{\epsilon}$\\
Grad-sub & - & - &  & $d\epsilon^{-4}$ & $d\epsilon^{-2}$ & & $d\epsilon^{-4}$ & $d\epsilon^{-2}$ & $d\frac{\log(\epsilon^{-1})}{\epsilon}$\\
MYULA \cite{durmus2022proximal} & \textcolor{gray}{$d^5\frac{\log^2(\epsilon^{-1})}{\epsilon^2}$} & - &  & \textcolor{gray}{$d^5\frac{\log^2(\epsilon^{-1})}{\epsilon^2}$} & - & & \textcolor{gray}{$d\log(d)\frac{\log^2(\epsilon^{-1})}{\epsilon^2}$} & - & - \\
SSGLD \cite{durmus2019analysis} & $\epsilon^{-4}$ & $\epsilon^{-2}$ &  & - & - &  & - & - & - \\
SPGLD \cite{durmus2019analysis} & - & - &  & \textcolor{gray}{$d\epsilon^{-4}$} & \textcolor{gray}{$d\epsilon^{-2}$} &  & \textcolor{gray}{$d\epsilon^{-4}$} & \textcolor{gray}{$d\epsilon^{-2}$} & \textcolor{gray}{$d\epsilon^{-2}$} \\
ASF \cite{liang2022proximal} & \textcolor{gray}{$\frac{\log(d^{1/2}\epsilon^{-1})}{\epsilon}$} & - &  & - & - &  & - & - & - \\
\bottomrule
        \end{tabular}}
\vspace*{0.2cm}       
\caption{Computational complexities of the proposed methods and existing methods in terms of the accuracy $\epsilon$ and the dimension $d$. The entries depict the order of the necessary number of iterations to reach accuracy $\epsilon$ to the target density in a specific metric. Gray entries indicate cases where it is necessary to iteratively compute the proximal mapping of $G\circ K$, or even $F+G\circ K$, which is not explicit in general for non-trivial $K$. %
MYULA is initialized with the Dirac distribution at an arbitrary point. ASF \cite{liang2022proximal} is initialized at a minimizer of the potential $U$ regularized by a quadratic term. For the remaining methods we assume initialization using a warm start such that the Wasserstein-2 distance to the target density of the initial distribution is bounded by a constant. In all cases $G$ is only assumed to be convex and Lipschitz continous. Here, SSGLD stands for Stochastic SubGradient Langevin Dynamics and SPGLD stands for Stochastic Proximal Gradient Langevin Dynamics}
\label{table:overview}
\end{table}
A limitation of our work is that, so far, we were not able to establish ergodicity and existence and uniqueness of a stationary measure for the transition kernel corresponding to one iteration of our algorithm. This has the disadvantage that, for sampling, several parallel Markov chains are required instead of a single chain from which successive samples can be used.  Existing literature typically relies on differentiability of the potential $U$ in order to show that a stationary measure for the transition kernel exists and that the corresponding Markov chain is geometrically ergodic \cite{durmus2017nonasymptotic,durmus2019analysis,roberts1996exponential,
pereyra2016proximal,durmus2022proximal}. However, due to the non-differentiability of the considered potentials and the fact that we do not consider a smoothed approximation of the potential $U$, an application of these results is not feasible in our setting.

\paragraph{Outline of the paper}
The remainder of the paper is organized as follows. After introducing the notation and listing some essential preliminary results in \Cref{sec:preliminaries}, we present the proposed algorithms in \Cref{sec:prox_subgrad,sec:grad_subgrad}. For each method we provide corresponding non-asymptotic convergence results in the respective section. Numerical experiments and a brief conclusion are further provided in \Cref{sec:experiments,sec:conclusion}, respectively.

\section{Notation and Preliminaries}\label{sec:preliminaries}
Slightly abusing notation, we write $\|\;\cdot\;\|$ for both the $2$-norm on $\R^d$ and the norm of a linear operator. A proper function $f:\R^d\rightarrow (-\infty,\infty]$ is called $m$-strongly convex if for any $x,y\in\R^d$ and $\lambda\in(0,1)$ 
\[f(\lambda x +(1-\lambda)y)\leq \lambda f(x) + (1-\lambda)f(y) - m\frac{\lambda(1-\lambda)}{2}\|x-y\|^2.\]
For a convex function $f:\R^d\rightarrow (-\infty,\infty]$ we define the subdifferential as the (possibly empty) set
\[
\partial f(x) = \left\{y\in\R^d\;\middle|\; \forall h\in\R^d:\; f(x)+\langle y, x+h\rangle\leq f(x+h)\right\},
\]
where $\langle \cdot , \cdot \rangle$ denotes that standard inner product between vectors in $\R^d$.
For proper, convex, and lsc $f$, $\partial f(x)\neq\emptyset$ for any $x\in (\dom (f))^\circ$, the interior of $\dom (f)$, where $\dom (f)\coloneqq\left\{x\in\R^d\;\middle|\;f(x)<\infty\right\}$.
We further define the proximal mapping of $f$ as
\begin{equation}\label{eq:definition_prox}
{\prox}_{f}(x) = \argmin_{z\in\R^d} f(z) + \frac{1}{2} \|x-z\|^2.
\end{equation}
If $f$ is proper, convex, and lsc, ${\prox}_{f}(x)$ is well-defined and single-valued for any $x$, that is, the minimizer in \eqref{eq:definition_prox} exists and is unique. We denote the Borel $\sigma$-algebra on $\R^d$ as $\Bc(\R^d)$ and the set of probability measures on the measurable space $(\R^d,\Bc(\R^d))$ as $\Pc(\R^d)$. A \emph{Markov transition kernel} is a function $S:\R^d\times \Bc(\R^d)\rightarrow [0,1]$ such that $S(\;\cdot\;,A)$ is measurable for any $A$ and $S(x,\;\cdot\;)$ is a probability measure for any $x$. We denote the application of the kernel $S$ to a probability measure $\mu\in\Pc(\R^d)$ as $\mu S$ defined as $\mu S(A) = \int S(x,A)\;\d\mu(x)$ which is again a probability measure. If there exists a \emph{density} $g$ such that $S(x,A) = \int_A g(x,y)\;\d y$, Fubini's theorem implies that $y\mapsto \int g(x,y)\;\d \mu(x)$ is a density of $\mu S$ and, therefore, we can compute the integral of a function $f$ with respect to $\mu S$ as $\int f(y)\;\d \mu S(y) = \int \int f(y) g(x,y)\;\d\mu(x) \d y$. Let $\Pc_2(\R^d)$ denote the set of all probability measures with finite second moment, that is, all $\mu\in\Pc(\R^d)$ such that $\int \|x\|^2 \; \d \mu(x)<\infty$. We write $\mu\ll\nu$ if $\mu$ is absolutely continuous with respect to $\nu$ and in this case the Radon-Nikod\'ym derivative of $\mu$ with respect to $\nu$ is denoted as $\frac{\d\mu}{\d\nu}$. By $(B_k)_k$ we always denote an i.i.d. sequence of $d$-dimensional standard Gaussian random variables. The Kullback-Leibler (KL) divergence between two measures $\mu,\nu\in\Pc(\R^d)$ is defined as
\begin{equation}
\KL(\mu|\nu) = \begin{cases}
\int \frac{\d \mu}{\d \nu}(x) \log(\frac{\d \mu}{\d \nu}(x)) \;\d \nu(x)\quad &\text{if } \mu \ll \nu\\
\infty \quad &\text{otherwise.}
\end{cases}
\end{equation}
A coupling of $\mu$ and $\nu$ is a probability measure $\zeta$ on $\R^d\times\R^d$ with marginal distributions $\mu$ and $\nu$, that is, for any $A\in\Bc(\R^d)$, $\zeta(A\times\R^d)=\mu(A)$ and $\zeta(\R^d\times A)=\nu(A)$. We denote the set of all couplings of $\mu$ and $\nu$ as $\Pi(\mu,\nu)$. In a slight abuse of terminology, we will also refer to a couple of random variables $(X,Y)$ as a coupling of $\mu$ and $\nu$ if there exists $\zeta\in\Pi(\mu,\nu)$ such that $(X,Y)$ is distributed according to $\zeta$. The (Kantorovich) Wasserstein 2-distance between two measures $\mu,\nu\in\Pc_2(\R^d)$ is defined as
\begin{equation}
W_2(\mu,\nu) = \left(\inf\limits_{\zeta\in\Pi(\mu,\nu)} \int \|x-y\|^2\;\d \zeta(x,y)\right)^{1/2}.
\end{equation}
By \cite{villani2009optimal} the infimum in the definition of the Wasserstein 2-distance is, in fact, attained, i.e., there exists a coupling $\zeta^*$ such that $W^2_2(\mu,\nu) = \int \|x-y\|^2\;\d \zeta^*(x,y) = \E[\|X-Y\|^2]$ for $(X,Y)\sim \zeta^*$. We will refer to $\zeta^*$, respectively $(X,Y)$, as an optimal coupling for $\mu,\nu$.

For a potential $V:\R^d\rightarrow [0,\infty)$ we define the potential energy and the Boltzman H-functional $\Ec_V,\Hc:\Pc_2(\R^d)\rightarrow [0,\infty]$ as
\begin{equation}
\begin{aligned}
\Ec_V(\mu) &= \int V(x)\; \d \mu(x),\quad\text{and}\quad
\Hc(\mu) &= \KL(\mu|\Leb),
\end{aligned}
\end{equation}
respectively, where Leb denotes the Lebesgue measure. Further, for the given potential $U$ from \eqref{eq:intro:distribution} we define the free energy functional $\Fc = \Hc + \Ec_U = \Hc + \Ec_F + \Ec_{G\circ K}$. The relevance of the free energy functional is evident from the following result.
\begin{lemma}\label{lem:F_diff_KL}
It holds that $\pi\in\Pc_2(\R^d)$, $\Ec_U(\pi)<\infty$, and $\Hc(\pi)<\infty$.
Moreover, $\mu\in\Pc_2(\R^d)$ with $\Ec_U(\mu)<\infty$ satisfies $\Fc(\mu)-\Fc(\pi)=\KL(\mu|\pi)$.
\end{lemma}
For a proof see \cite[Lemma 1]{durmus2019analysis}. This result implies that $\Fc(\mu)-\Fc(\pi)\geq 0$ with equality if and only if $\mu=\pi$. It further allows us to formulate the task of sampling as a minimization problem. Namely, if $(X_k)_k$ denotes the MC used to approximate the distribution $\pi$, and $\mu_k$ the distribution of $X_k$, then convergence of the MC is ensured if $\mu_k$ minimizes $\mu\mapsto \Fc(\mu)-\Fc(\pi)$ in the limit.

\section{Proximal-subgradient Langevin Sampling}\label{sec:prox_subgrad}
In this section we investigate the first of the proposed sampling methods, namely the Proximal-subgradient Langevin algorithm (Prox-sub) which is depicted in \Cref{algo:prox_subgrad}. One iteration of the method consists of a subgradient step with respect to $G\circ K$ and a proximal step with respect to $F$ followed by adding the Gaussian random variable $B_{k+1}$. In particular, the method is feasible in the case that both $F$ and $G$ are non-differentiable. At this point it might seem unusual to apply different step sizes $\tau_k$ and $\tau_{k+1}$ within a single iteration of \Cref{algo:prox_subgrad}. In the way the algorithm is presented, the step size of proximal step with respect to $F$ in iteration $k$ is identical to the step size of the subgradient update for $G\circ K$ in iteration $k+1$ which will turn out to be favorable within the proofs of convergence, see \Cref{remark:step_sizes}.

\begin{algorithm}
\setstretch{1.15}
\caption{Proximal-subgradient Langevin Algorithm (Prox-sub)}\label{algo:prox_subgrad}
\textbf{Input:} Initialization $X_0$, number of iterations $K$, step sizes $(\tau_k)_{k=0}^K$, $\tau_k>0$.\\ \vspace*{-\baselineskip}
\begin{algorithmic}[1]
\FOR{$k=0,1,2,\dots,K-1$}
\STATE Pick $Y_{k+1} \in \partial G(KX_k)$
\STATE $X_{k+1} = {\prox}_{\tau_{k+1} F}(X_k - \tau_{k} K^* Y_{k+1}) + \sqrt{2\tau_{k+1}}B_{k+1}$
\ENDFOR
\end{algorithmic}
\textbf{Output:} $X_K$.
\end{algorithm}
Let $\mu_k$ denote the distribution of $X_k$. In accordance with the algorithm, we define the following Markov transition kernels
\begin{equation}
\begin{cases}
S^G_{\tau}(x,A) = \delta_{x - \tau K^*\theta(Kx)}(A)\\
S^F_{\tau}(x,A) = \delta_{{\prox}_{\tau F}(x)}(A)\\
T_{\tau}(x,A) = (4\pi\tau)^{-d/2} \int\limits_A \exp\left(-\left\| x - z \right\|^2\middle/(4\tau)\right) \; \d z
\end{cases}
\end{equation}
where $\delta_x\in\Pc_2(\R^d)$ denotes the Dirac measure concentrated in $x$ and $\theta:\R^d\rightarrow \R^d$ is a measurable function such that for any $z\in\R^{d'}$, $\theta(z)\in\partial G(z)$.
We introduce the notation $R_{\tilde{\tau},\tau} \coloneqq S^G_{\tilde{\tau}}S^F_{\tau} T_{\tau}$ so that the iteration on the distributions $(\mu_k)_k$ induced by \Cref{algo:prox_subgrad} can be written as $\mu_{k+1}=\mu_k R_{\tau_k,\tau_{k+1}}$. Note that $R_{\tilde{\tau},\tau} \coloneqq S^G_{\tilde{\tau}}S^F_{\tau} T_{\tau}$ (and similar for other transition kernels) is purely symbolic notation that only makes sense in conjunction with the application to a probability measure $\mu$, where for the latter we use the order $\mu R_{\tilde{\tau},\tau} = ((\mu S^G_{\tilde{\tau}})S^F_{\tau}) T_{\tau}$. 

\subsection{Convergence in a General Setting}
The main results of this section will be \Cref{thm:conv_estimate_general,cor:prox_sub_conv_weak} where we show convergence of the running mean of the sample distributions in KL to the target density.
As mentioned above, the existence and uniqueness of a stationary measure for the kernel $R_{\tilde{\tau},\tau}$ is still an open question. Instead of proving ergodicity of the Markov chain generated by $R_{\tilde{\tau},\tau}$, thus, the strategy for proving convergence of the proposed algorithm is to directly bound the discrepancy between $\mu_k$ and $\pi$ in a suitable metric. To achieve this, we closely follow \cite{durmus2019analysis}, that is, we will derive a bound for $\Fc(\mu_k)-\Fc(\pi)=\KL(\mu_k|\pi)$ along the iteration $\mu_{k+1}=\mu_k R_{\tau_k,\tau_{k+1}}$. In order to do so we will make use of the decomposition
\begin{equation}
\begin{aligned}
\Fc(\mu R_{\tilde{\tau},\tau})-\Fc(\pi)
= \Ec_F(\mu S^G_{\tilde{\tau}}S^F_{\tau}T_{\tau})-\Ec_F(\mu S^G_{\tilde{\tau}}S^F_{\tau})
+\Ec_F(\mu S^G_{\tilde{\tau}}S^F_{\tau}) -\Ec_F(\pi) \\
+\Ec_{G\circ K}(\mu S^G_{\tilde{\tau}}S^F_{\tau}T_{\tau}) - \Ec_{G\circ K}(\pi)
+\Hc(\mu S^G_{\tilde{\tau}}S^F_{\tau}T_{\tau}) - \Hc(\pi)
\end{aligned}
\end{equation}
where we will derive estimates for each difference separately.
\begin{lemmaE}[][end]
\label{lem:operator_T}
Let $\mu\in\Pc_2(\R^d)$. If $F$ is $L_F$-Lipschitz continuous it holds true that $\Ec_F(\mu T_{\tau}) - \Ec_F(\mu) \leq L_F\sqrt{2d\tau}$, and if $F$ is differentiable with $L_{\nabla F}$-Lipschitz gradient, then $\Ec_F(\mu T_{\tau}) - \Ec_F(\mu) \leq \tau L_{\nabla F} d$.
\end{lemmaE}
\begin{proofE}
If $F$ is Lipschitz continuous a simple computation yields
\begin{equation}
\begin{aligned}
\Ec_F(\mu T_\gamma) - \Ec_F(\mu)=\int F(x) \;\d(\mu T_{\tau})(x) - \int F(x) \;\d\mu(x)\\
= (4\pi\tau)^{-d/2} \int\int (F(x+y)-F(x)) e^\frac{-\|y\|^2}{4\tau} \;\d \mu(x)\d y\\
\leq (4\pi\tau)^{-d/2} \int\int L_F \|y\| e^\frac{-\|y\|^2}{4\tau} \;\d \mu(x)\d y \\\leq \left((4\pi\tau)^{-d/2} \int L_F^2 \|y\|^2 e^\frac{-\|y\|^2}{4\tau} \;\d y\right)^{\frac{1}{2}}= L_F\sqrt{d 2 \tau}.
\end{aligned}
\end{equation}
If $F$ is differentiable with Lipschitz gradient the fundamental theorem of calculus leads to
\begin{equation}
\begin{aligned}
F(y)-F(x) - \langle \nabla F(x), y-x\rangle 
 &= \int\limits_0^1 \langle \nabla F(x+t(y-x)) - \nabla F(x), y-x\rangle \;\d t\\ 
 & \leq  \int\limits_0^1  L_{\nabla F}t\|y-x\|^2 \;\d t\leq \frac{L_{\nabla F}}{2}\|y-x\|^2.
\end{aligned}
\end{equation}
Therefore we can estimate
\begin{equation}
\begin{aligned}
\Ec_F(\mu T_\gamma) - \Ec_F(\mu)= (4\pi\tau)^{-d/2} \int\int (F(x+y)-F(x)) e^\frac{-\|y\|^2}{4\tau} \;\d \mu(x)\d y\\
\leq (4\pi\tau)^{-d/2} \int\int \left\{ \frac{L_{\nabla F}}{2} \|y\|^2 + \langle \nabla F(x), y \rangle \right\} e^\frac{-\|y\|^2}{4\tau} \;\d \mu(x)\d y = \tau L_{\nabla F} d
\end{aligned}
\end{equation}
where in the last equality we used the fact that for any $x\in\R^d$ $\int \langle \nabla F(x), y \rangle e^\frac{-\|y\|^2}{4\tau} \;\d y=0$.
\end{proofE}
\begin{lemmaE}{(\cite[Lemma 29]{durmus2019analysis})}\label{lem:SF_general}
For any $\mu,\nu\in\Pc_2(\R^d)$ it holds $2\tau \left\{\Ec_F(\mu S^F_\tau) - \Ec_F(\nu)\right\}
\leq W^2_2(\mu,\nu) - W^2_2(\mu S_\tau^F, \nu)$.
\end{lemmaE}
\begin{proofE}
We denote $z={\prox}_{\tau F}(x)$. The definition of the proximal mapping implies
\[0\in\tau\partial F(z) + z-x.\]
Thus, $\tau^{-1}(x-z)\in\partial F(z)$ implying for any $w\in\R^d$
\begin{equation}
\begin{aligned}
2\tau\left\{ F(z)-F(w)\right\} \leq 2\langle x-z, z-w\rangle \leq \|x-w\|^2-\|z-w\|^2
\end{aligned}
\end{equation}
where for the last inequality we used the identity $\|x-w\|^2 = \|x-z\|^2 + 2\langle x-z,z-w\rangle + \|z-w\|^2$. Plugging in an optimal coupling $(X,W)$ for the distributions $(\mu,\nu)$ and taking the expectation yields the desired result, since $W^2_2(\mu S_\tau^F, \nu)\leq\E[\|Z-W\|^2]$ for any coupling $(Z,W)$ of $(\mu S_\tau^F, \nu)$.
\end{proofE}
\begin{lemmaE}{(\cite[Lemma 26]{durmus2019analysis})}
\label{lem:subgrad}
Let \Cref{ass:general}, \ref{ass:G} hold. Then for any $\mu,\nu\in\Pc_2(\R^d)$
\begin{equation}
\begin{aligned}
2\tau \left\{ \Ec_{G\circ K}(\mu) - \Ec_{G\circ K}(\nu) \right\}\leq W^2_2(\mu,\nu) -W^2_2(\mu S^G_{\tau},\nu) + \tau^2 L_G^2\|K\|^2.
\end{aligned}
\end{equation}
\end{lemmaE}
\begin{proofE}
Let $z = x - \tau K^*\theta(x)$. Then $\tau^{-1}(x-z)\in\partial (G\circ K) (x)$ and, thus,
\[G(Kx) + \langle \tau^{-1}(x-z), w-x\rangle \leq G(Kw)\]
for any $w$. Together with the fact that $\|z-w\|^2 = \|z-x\|^2 + 2\langle z-x,x-w\rangle + \|w-x\|^2$ we obtain
\begin{equation}
\begin{aligned}
2\tau (G(Kx) - G(Kw))\leq 2\langle x-z, x-w\rangle =  \|z-x\|^2 + \|w-x\|^2 - \|z-w\|^2\\ =
\|w-x\|^2 - \|z-w\|^2 + \tau^2 \|K^*\theta(x)\|^2\\ \leq
\|w-x\|^2 - \|z-w\|^2 + \tau^2 L_G^2\|K\|^2\\
\end{aligned}
\end{equation}
where for the last inequality we use the fact that elements of $\partial G$ are bounded by $L_G$ due to Lipschitz continuity of $G$. Plugging in an optimal coupling $(X,W)$ of the distributions $\mu$ and $\nu$ and taking the expectation yields
\begin{equation}
\begin{aligned}
2\tau \left\{ \Ec_{G\circ K}(\mu) - \Ec_{G\circ K}(\nu) \right\}\leq W^2_2(\mu,\nu) -W^2_2(\mu S^G_{\tau},\nu) + \tau^2 L_G^2\|K\|^2
\end{aligned}
\end{equation}
where we used that $W^2_2(\mu S^G_{\tau},\nu)\leq \E[\|Z-W\|^2]$.
\end{proofE}

\begin{lemma}\label{lem:H_inequality}
Let $\mu,\nu\in\Pc_2(\R^d)$, $\Hc(\nu)<\infty$, then $2\tau\left\{ \Hc(\mu T_\tau)-\Hc(\nu)\right\}\leq W_2^2(\mu,\nu)-W_2^2(\mu T_\tau,\nu)$.
\end{lemma}
For a proof see \cite[Lemma 5]{durmus2019analysis}.

We are now in the position to collect the results in order to obtain a bound on the free energy functional.
\begin{proposition}\label{lem:KL_ineq_general}
Let \Cref{ass:general} be satisfied. Then for any $\mu \in \Pc_2(\R^d)$
\begin{equation}\label{eq:adding_prox_subgrad}
\begin{aligned}
2\tau\left\{ \Fc(\mu R_{\tilde{\tau},\tau})-\Fc(\pi)\right\}\leq
W^2_2(\mu S^G_{\tilde{\tau}},\pi) - W^2_2(\mu R_{\tilde{\tau},\tau}S^G_{\tau},\pi) + 2\tau \phi(\tau) + \tau^2 L_G^2\|K\|^2
\end{aligned}
\end{equation}
where $\phi(\tau)=L_F\sqrt{2d\tau}$ or $\phi(\tau)=\tau L_{\nabla F} d$ if \Cref{ass:general}, \ref{ass:F}, \ref{ass:Fa} or \ref{ass:Fb} is satisfied, respectively.
\end{proposition}
\begin{proof}
From \Cref{lem:operator_T,lem:SF_general,lem:subgrad,lem:H_inequality} we obtain the following inequalities
\begin{equation}
\label{eq:pd_ex}
\begin{aligned}
2\tau\left\{ \Ec_F(\mu R_{\tilde{\tau},\tau}) - \Ec_F(\mu S^G_{\tilde{\tau}} S^F_\tau)\right\}&\leq 2\tau \phi(\tau),\\
2\tau \left\{\Ec_F(\mu S^G_{\tilde{\tau}} S^F_\tau) - \Ec(\pi)\right\}
&\leq W^2_2(\mu S^G_{\tilde{\tau}},\pi) - W^2_2(\mu S^G_{\tilde{\tau}} S^F_\tau, \pi),\\
2\tau \left\{ \Ec_{G\circ K}(\mu R_{\tilde{\tau},\tau}) - \Ec_{G\circ K}(\pi) \right\}&\leq W^2_2(\mu R_{\tilde{\tau},\tau},\pi) - W^2_2(\mu R_{\tilde{\tau},\tau}S^G_{\tau},\pi) + \tau^2 L_G^2\|K\|^2,\\
2\tau\left\{ \Hc(\mu R_{\tilde{\tau},\tau})-\Hc(\pi)\right\}&\leq W_2^2(\mu S^G_{\tilde{\tau}} S^F_\tau,\pi)-W_2^2(\mu R_{\tilde{\tau},\tau},\pi).
\end{aligned}
\end{equation}
Adding those yields the desired result.
\end{proof}
Unfortunately, convergence cannot be proven directly for the sequence of distributions $\mu_k$ generated from \Cref{algo:prox_subgrad}, but only for a convex combination of these distributions. We define for $\lambda_k>0$, $\tau_k>0$, a burn in time $N\in\N$, and an initial distribution $\mu\in\Pc_2(\R^d)$
\begin{equation}\label{eq:running_average}
\Lambda_{N,N+n} = \sum\limits_{k=N+1}^{N+n}\lambda_k,\quad Q^k_{\tau} = R_{\tau_1,\tau_2}R_{\tau_2,\tau_3}\dots R_{\tau_k, \tau_{k+1}},\quad \nu_n^N = \Lambda_{N,N+n}^{-1}\sum\limits_{k=N+1}^{N+n}\lambda_k \mu Q^k_{\tau}.
\end{equation}
with the convention that $\mu Q^0_{\tau}=\mu$.
\begin{theorem}\label{thm:conv_estimate_general}
Select the stepsizes $(\tau_k)_k$ and weights $(\lambda_k)_k$ such that, for all $k$, $\frac{\lambda_{k+1}}{\tau_{k+2}} \leq \frac{\lambda_k}{\tau_{k+1}}$, and let \cref{ass:general} hold. Then
\begin{equation}
\begin{aligned}
\KL(\nu_n^N | \pi)\leq \Lambda_{N,N+n}^{-1} \frac{\lambda_{N+1}}{2\tau_{N+2}} W^2_2(\mu Q^{N}_{\tau} S^G_{\tau_{N+1}},\pi)\\
+\Lambda_{N,N+n}^{-1} \sum\limits_{k=N+1}^{N+n}\lambda_k\left\{\phi(\tau_{k+1}) + \frac{1}{2}\tau_{k+1} L_G^2\|K\|^2\right\}
\end{aligned}
\end{equation}
where $\phi(\tau)=L_F\sqrt{2d\tau}$ or $\phi(\tau)=\tau L_{\nabla F} d$ if \Cref{ass:general}, \ref{ass:F}, \ref{ass:Fa} or \ref{ass:Fb} is satisfied, respectively.
\end{theorem}
\begin{proof}
Using \Cref{lem:KL_ineq_general} and convexity of the Kullback-Leibler divergence we compute
\begin{equation}\label{eq:telescope}
\begin{aligned}
\KL(\nu_n^N | \pi)\leq \Lambda_{N,N+n}^{-1} \sum\limits_{k=N+1}^{N+n} \lambda_k \KL(\mu Q^k_{\tau}|\pi) \\
\leq \Lambda_{N,N+n}^{-1} \sum\limits_{k=N+1}^{N+n}\frac{\lambda_k}{2\tau_{k+1}}\left\{W^2_2(\mu Q^{k-1}_{\tau} S^G_{\tau_k},\pi) - W^2_2(\mu Q^{k}_{\tau}S^G_{\tau_{k+1}},\pi)\right\} \\
+\Lambda_{N,N+n}^{-1} \sum\limits_{k=N+1}^{N+n}\lambda_k\left\{\phi(\tau_{k+1}) + \frac{1}{2}\tau_{k+1} L_G^2\|K\|^2\right\}\\
\leq \Lambda_{N,N+n}^{-1} \frac{\lambda_{N+1}}{2\tau_{N+2}} W^2_2(\mu Q^{N}_{\tau} S^G_{\tau_{N+1}},\pi)\\
+\Lambda_{N,N+n}^{-1} \sum\limits_{k=N+1}^{N+n-1} \underbrace{\left\{\frac{\lambda_{k+1}}{2\tau_{k+2}}-\frac{\lambda_k}{2\tau_{k+1}}\right\}}_{\leq 0} W^2_2(\mu Q^{k}_{\tau}S^G_{\tau_{k+1}},\pi)\\
+\Lambda_{N,N+n}^{-1} \sum\limits_{k=N+1}^{N+n}\lambda_k\left\{\phi(\tau_{k+1}) + \frac{1}{2}\tau_{k+1} L_G^2\|K\|^2\right\}\\
\leq \Lambda_{N,N+n}^{-1} \frac{\lambda_{N+1}}{2\tau_{N+2}} W^2_2(\mu Q^{N}_{\tau} S^G_{\tau_{N+1}},\pi)
+\Lambda_{N,N+n}^{-1} \sum\limits_{k=N+1}^{N+n}\lambda_k\left\{\phi(\tau_{k+1}) + \frac{1}{2}\tau_{k+1} L_G^2\|K\|^2\right\}
\end{aligned}
\end{equation}
concluding the proof.
\end{proof}
\begin{remark}\label{remark:step_sizes}
Note that the proof of \Cref{thm:conv_estimate_general} makes it clear why we apply different step sizes $\tau_k$ and $\tau_{k+1}$ during a single iteration of \Cref{algo:prox_subgrad}. The basis for the convergence proof is the estimate \eqref{eq:adding_prox_subgrad}. In order to obtain the elegant telescope sum in \eqref{eq:telescope} based on this estimate, we have to choose the Markov kernels as $R_{\tau_k,\tau_{k+1}}$ which couples the step sizes of consecutive updates $R_{\tau_k,\tau_{k+1}}$ and $R_{\tau_{k+1},\tau_{k+2}}$. The remaining degree of freedom regarding the step sizes is to choose different values for the $F$ and the $G$ update within a single iteration. We can also observe that, by the change of variables $Z_k\coloneqq X_k - \tau_kK^*Y_{k+1}$, \Cref{algo:prox_subgrad} can be rewritten as
\begin{equation}
\begin{cases}
Z_{k+\frac{1}{2}} = {\prox}_{\tau_{k+1}F}(Z_k) + \sqrt{2\tau_{k+1}}B_{k+1}\\
Y_{k+1} \in \partial G(K Z_{k+\frac{1}{2}})\\
Z_{k+1} = Z_{k+\frac{1}{2}} - \tau_{k+1}K^*Y_{k+1}
\end{cases}
\end{equation}
where now the same step size is used for both the proximal and the subgradient update. However, this comes at the cost, that convergence is not obtained directly for the variable $Z_k$, but only for $X_k$ as shown above.
\end{remark}
Choosing a constant stepsize $\tau_k=\tau$ and $\lambda_k=1$ for all $k$, \Cref{thm:conv_estimate_general} simplifies to 
\[\KL(\nu_n^N | \pi)\leq \frac{1}{2n\tau} W^2_2(\mu Q^{N}_{\tau} S^G_{\tau},\pi)+
\left\{\phi(\tau) + \frac{1}{2}\tau L_G^2\|K\|^2\right\}.\]
That is, to reach accuracy $\epsilon>0$ in $\KL$ the computational complexity results in $O(\epsilon^{-3})$ in the case that $F$ is Lipschitz continuous and $O(\epsilon^{-2})$ in the case that $F$ has Lipschitz gradient. More precisely, we obtain the following corollary.
\begin{corollary}\label{cor:prox_sub_conv_weak}
Let \cref{ass:general} hold, $N\in\N$ be a burn-in time and $\epsilon>0$ arbitrary. Choose $\lambda_k=1$ and $\tau_k=\tau_\epsilon>0$ fixed for all $k$ such that
\begin{equation}\label{eq:prox_grad_weak}
\begin{aligned}
\begin{cases}
L_F\sqrt{2d\tau_\epsilon} + \frac{1}{2}\tau_{\epsilon} L_G^2\|K\|^2 &< \frac{\epsilon}{2}\quad \text{under \Cref{ass:general}, \ref{ass:F}, \ref{ass:Fa}}\\
\tau_\epsilon \left(L_{\nabla F} d + \frac{1}{2}L_G^2\|K\|^2\right) &< \frac{\epsilon}{2}\quad \text{under \Cref{ass:general}, \ref{ass:F}, \ref{ass:Fb}}.
\end{cases}
\end{aligned}
\end{equation}
Then with $n_\epsilon\in\N$ such that $n_\epsilon>\frac{1}{\epsilon\tau_{\epsilon}} W^2_2(\mu Q^{N}_{\tau} S^G_{\tau_{\epsilon}},\pi)$, for any $n\geq n_\epsilon$, it holds $\KL(\nu_n^N | \pi)<\epsilon$.
\end{corollary}
\begin{remark}\
\begin{itemize}
\item In the case of \Cref{ass:general}, \ref{ass:F}, \ref{ass:Fa}, the subgradient method from \cite[Section 4.1]{durmus2019analysis} also constitutes an applicable method with similar convergence guarantees. The case of \Cref{ass:general}, \ref{ass:F}, \ref{ass:Fb}, however, yields applicability of our method also to situations where $F$ is not Lipschitz continuous (e.g. the important case where $F$ is an $L^2$ discrepancy), a setting where no convergence guarantees for the subgradient method are currently available.

\item In terms of convergence rates, the results of the following section, dealing with higher regularity of $F$, constitute a second, significant improvement upon existing results. This is due to the fact that proximal-gradient methods such as the one from \cite[Section 4.2]{durmus2019analysis}, which would yield similar convergence rates, are not feasible for the setting at hand, since we assume the non-differentiability to be contained within $G$. In order to deploy the proximal-gradient method from \cite[Section 4.2]{durmus2019analysis} one would have to perform the proximal step with respect to $G\circ K$ which is, in general, not explicit.

\item Note that convergence in KL immediately implies convergence in the TV norm via Pisnker's inequality.
\end{itemize}
\end{remark}
\subsection{Improved Convergence for Smooth $F$}\label{sec:convergence_strong}
We will now improve our convergence results by imposing stronger conditions on $F$ without additionally constraining $G$. This will lead to convergence of the distributions $\mu Q^k_{\tau} S^G_{\tau_{k}}$ in Wasserstein distance directly rather than the average distributions $\nu_k^N$ as shown in \Cref{thm:convergence_strong,cor:prox_grad_w2}.
\begin{ass}\label{ass:strong} The functionals $F$ and $G$ satisfy the following regularity assumptions:
\begin{enumerate}[(1)]
\item $F$ is $m$-strongly convex and differentiable with Lipschitz continuous gradient with constant $L_{\nabla F}$.\label{ass:F_strong}
\item $G$ is convex and Lipschitz continuous with constant $L_G$.\label{ass:G_strong}
\end{enumerate}
\end{ass}
At first, we derive an improved bound on the proximal step as follows.
\begin{lemmaE}\label{lem:S1_strong}
Let \Cref{ass:strong}, \ref{ass:F} hold and select $\tau>0$ such that $\tau \leq \frac{m}{2L_{\nabla F}^2-m^2}$. Then for any $\mu,\nu\in\Pc_2(\R^d)$
\begin{equation}
2\tau \left\{\Ec_F(\mu S^F_\tau) - \Ec_F(\nu)\right\}
\leq (1- \frac{m\tau}{2}) W^2_2(\mu,\nu) - W^2_2(\mu S_\tau^F, \nu)
\end{equation}
\end{lemmaE}
\begin{proofE}
We denote $z={\prox}_{\tau F}(x)$. Since $F$ is $m$-strongly convex and has Lipschitz gradient, as shown in \cite[Lemma B.2]{valkonen2020testing}, we have for any $\tilde{x},\tilde{y},\tilde{z}\in \R^d$ and $\alpha>0$
\begin{equation}
F(\tilde{x}) - F(\tilde{y}) \leq \langle\nabla F(\tilde{z}),\tilde{x}-\tilde{y}\rangle - \frac{m-\alpha L_{\nabla F}^2}{2}\|\tilde{x}-\tilde{y}\|^2 + \frac{1}{2\alpha}\|\tilde{x}-\tilde{z}\|^2.
\end{equation}
In particular, setting $\tilde{x} = z,\tilde{y} = x,\tilde{z} = z$ and noticing that, since $z={\prox}_{\tau F}(x)$ it holds true that $\nabla F(z) = \tau^{-1}(x-z)$, we obtain
\begin{equation}
\begin{aligned}
F(z) - F(x) \leq \langle\nabla F(z),z-x\rangle - \frac{m-\alpha L_{\nabla F}^2}{2}\|z-x\|^2 + \frac{1}{2\alpha}\|z-z\|^2\\ = 
-\tau^{-1} \|z-x\|^2 - \frac{m-\alpha L_{\nabla F}^2}{2}\|z-x\|^2.
\end{aligned}
\end{equation}
Letting $\alpha\rightarrow 0$, this becomes $F(z) - F(x) \leq -(\tau^{-1}+\frac{m}{2})\|z-x\|^2.$ Setting, $\tilde{x} = x,\tilde{y} = w,\tilde{z} = z$ on the other hand yields 
\begin{equation}
F(x) - F(w) \leq \langle\nabla F(z),x-w\rangle - \frac{m-\alpha L_{\nabla F}^2}{2}\|x-w\|^2 + \frac{1}{2\alpha}\|x-z\|^2.
\end{equation}
Therefore we find
\begin{equation}
\begin{aligned}
2\tau \left\{F(z)-F(w)\right\} = 2\tau \left\{F(z) - F(x) + F(x) - F(w)\right\}\\\leq
-(2+m\tau)\|z-x\|^2 + 2\tau\langle\nabla F(z),x-w\rangle - \tau(m-\alpha L_{\nabla F}^2)\|x-w\|^2 + \frac{\tau}{\alpha}\|x-z\|^2\\ 
\underbrace{- \| z-w \|^2 + \|z-x\|^2 + \|x-w\|^2 +2\tau \langle \tau^{-1}(z-x), x-w\rangle }_{=0}\\ =
(\frac{\tau}{\alpha}-1-m\tau)\|z-x\|^2 + 2\tau\langle\nabla F(z),x-w\rangle + (1- \tau(m-\alpha L_{\nabla F}^2))\|x-w\|^2 \\
- \| z-w \|^2 +2\tau \langle -\nabla F(z), x-w\rangle
\leq (1- \frac{m\tau}{2})\|x-w\|^2 - \| z-w \|^2
\end{aligned}
\end{equation}
where we chose $\alpha = \frac{\tau}{1+m\tau}$ and employed the assumption on $\tau$. By plugging in an optimal coupling $(X,W)$ for the distributions $(\mu,\nu)$ and taking the expectation we find
\[ 2\tau \left\{\Ec_F(\mu S^F_\tau) - \Ec_F(\nu)\right\}
\leq (1- \frac{m\tau}{2}) W^2_2(\mu,\nu) - W^2_2(\mu S_\tau^F, \nu).\]
\end{proofE}
\begin{remark}\
\begin{itemize}
\item Note that it always holds true that $m\leq L_{\nabla F}$ since $m$-strong convexity of $F$ is equivalent to
\[m\|x-y\|^2\leq \langle\nabla F(x)-\nabla F(y),x-y\rangle,\quad \text{for any $x,y$}.\]
Using the Lipschitz continuity of $\nabla F$ we find $m\|x-y\|^2\leq L_{\nabla F}\|x-y\|^2$. Therefore the upper bound on $\tau$ is always a positive real number.
\item It always holds true that $(1- \frac{m\tau}{2})>0$ since
\[\tau\leq\frac{m}{2L_{\nabla F}^2-m^2}\leq \frac{m}{2m^2-m^2} = \frac{1}{m}< \frac{2}{m}.\]
\end{itemize}
\end{remark}
Consequently, this leads to a stronger estimate on the free energy functional.
\begin{proposition}\label{lem:KL_ineq_strong}
Let \Cref{ass:strong} be satisfied and select $\tau>0$ such that $\tau\leq \frac{m}{2L_{\nabla F}^2-m^2}$. Then for any $\mu \in \Pc_2(\R^d)$
\begin{equation}
\begin{aligned}
2\tau\left\{ \Fc(\mu R_{\tilde{\tau},\tau})-\Fc(\pi)\right\} \leq
(1 - \frac{m\tau}{2}) W^2_2(\mu S^G_{\tilde{\tau}},\pi) - W^2_2(\mu R_{\tilde{\tau},\tau}S^G_{\tau},\pi)\\ + 2 L_{\nabla F} d\tau^2+  \tau^2 L_G^2\|K\|^2 
\end{aligned}
\end{equation}
\end{proposition}
\begin{proof}
The result is obtained analogously to the proof of \Cref{lem:KL_ineq_general} but by making use of the improved estimate from \Cref{lem:S1_strong} instead of \Cref{lem:SF_general}.
\end{proof}
In this case, we can deduce a convergence result without relying on a running average over the iterates.
\begin{theorem}\label{thm:convergence_strong}
Let \Cref{ass:strong} be satisfied and select the stepsizes $(\tau_k)_k$ such that $\tau_k\leq\frac{m}{2L_{\nabla F}^2-m^2}$ for all $k$. Then for any $\mu \in \Pc_2(\R^d)$
\begin{equation}
\begin{aligned}
W^2_2(\mu Q^k_{\tau} S^G_{\tau_{k+1}},\pi)\leq \prod\limits_{j=2}^{k+1}(1 - \frac{m\tau_j}{2}) W^2_2(\mu S^G_{\tau_1},\pi) 
+ \left(2 L_{\nabla F} d + L_G^2\|K\|^2\right)\sum\limits_{j=2}^{k+1} \tau_{j}^2 \prod\limits_{i=j+1}^{k+1}(1 - \frac{m\tau_i}{2}).
\end{aligned}
\end{equation}
In particular, for constant $\tau_j=\tau$ it holds true that
\begin{equation}
\begin{aligned}
W^2_2(\mu Q^k_{\tau} S^G_{\tau},\pi)\leq
(1 - \frac{m\tau}{2})^k W^2_2(\mu S^G_{\tau},\pi)
+ \left(2 L_{\nabla F} d + L_G^2\|K\|^2\right)\frac{2}{m}\tau
\end{aligned}
\end{equation}
\end{theorem}
\begin{proof}
Setting $\mu \mapsto \mu Q^{k-1}_{\tau}$, $\tilde{\tau}=\tau_k$, and $\tau=\tau_{k+1}$ in \Cref{lem:KL_ineq_strong}, the non-negativity of the Kullback-Leibler divergence together with \Cref{lem:F_diff_KL} implies
\begin{equation}
\begin{aligned}
W^2_2(\mu Q^k_{\tau} S^G_{\tau_{k+1}},\pi)\leq
(1 - \frac{m\tau_{k+1}}{2}) W^2_2(\mu Q^{k-1}_{\tau} S^G_{\tau_k},\pi) 
+ \left(2 L_{\nabla F} d + L_G^2\|K\|^2\right) \tau_{k+1}^2.
\end{aligned}
\end{equation}
Iterating over $k$ yields the first assertion. In particular, fixing $\tau_k = \tau$ for all $k$, we obtain
\begin{equation}
\begin{aligned}
W^2_2(\mu Q^k_{\tau} S^G_{\tau},\pi)\leq (1 - \frac{m\tau}{2})^k W^2_2(\mu S^G_{\tau},\pi)
+ \sum\limits_{j=2}^{k+1} \left(2 L_{\nabla F} d + L_G^2\|K\|^2\right)\tau^2 \prod\limits_{i=j+1}^{k+1}(1 - \frac{m\tau}{2})\\=
(1 - \frac{m\tau}{2})^k W^2_2(\mu S^G_{\tau},\pi)
+ \left(2 L_{\nabla F} d + L_G^2\|K\|^2\right)\tau^2\sum\limits_{j=2}^{k+1} (1 - \frac{m\tau}{2})^{k-j+1}\\=
(1 - \frac{m\tau}{2})^k W^2_2(\mu S^G_{\tau},\pi)
+ \left(2 L_{\nabla F} d + L_G^2\|K\|^2\right)2\tau\frac{1-(1 - \frac{m\tau}{2})^k}{m}\\ \leq
(1 - \frac{m\tau}{2})^k W^2_2(\mu S^G_{\tau},\pi)
+ \left(2 L_{\nabla F} d + L_G^2\|K\|^2\right)\frac{2}{m}\tau
\end{aligned}
\end{equation}
concluding the proof.
\end{proof}
In the setting of constant step size, \Cref{thm:convergence_strong} implies that in order to reach accuracy $\epsilon>0$ in squared Wasserstein 2-distance, $O(\epsilon^{-1} \log(\epsilon^{-1}))$ iterations will be needed. More precisely, we obtain the following non-asymptotic result.
\begin{corollary}\label{cor:prox_grad_w2}
Let \Cref{ass:strong} hold and $\tau_k=\tau_\epsilon$ for all $k$ and $n_\epsilon\in\N$ be such that
\begin{equation}\label{eq:stepsize_prox_grad_strong}
\tau_\epsilon<\min\left\{\frac{m\epsilon}{4}\left(2 L_{\nabla F} d + L_G^2\|K\|^2\right)^{-1}, \frac{m}{2L_{\nabla F}^2-m^2}\right\},\quad n_\epsilon>\frac{\log\left(\frac{\epsilon}{2W^2_2(\mu S^G_{\tau},\pi)}\right)}{\log\left(1 - \frac{m\tau_\epsilon}{2}\right)}.
\end{equation}
Then for any $n\geq n_\epsilon$, it holds $W^2_2(\mu Q^n_{\tau} S^G_{\tau},\pi)<\epsilon$.
\end{corollary}
As in \cite{ehrhardt2023proximal} we can show that a clever, non-constant step size leads to direct convergence without a remaining bias. For the proof we refer the reader to \cite[Theorem 3.3]{ehrhardt2023proximal}.
\begin{corollaryE}\label{cor:decreasing_step_size}
Let \Cref{ass:strong} hold. If the step sizes $(\tau_j)_j$ satisfy
\begin{enumerate}
\item $\lim\limits_{j\rightarrow \infty}\tau_j=0$, $\sum\limits_{j =1}^\infty\tau_j=\infty$,
\item $\frac{\tau_j}{1+\frac{m\tau_j}{2}}\leq \tau_{j+1}\leq\tau_j\leq \tau_1\leq \frac{m}{2L_{\nabla F}^2-m^2}$,
\item $\limsup\limits_{k\rightarrow\infty}\sum\limits_{j=2}^{k+1} \tau_{j} \prod\limits_{i=j+1}^{k+1}(1 - \frac{m\tau_i}{2})<\infty$,
\end{enumerate}
then it holds true that $W^2_2(\mu Q^k_{\tau} S^G_{\tau_{k+1}},\pi)\rightarrow 0,$ as $k\rightarrow \infty$.
\end{corollaryE}
An explicit sequence of step sizes satisfying the conditions of \Cref{cor:decreasing_step_size} is also provided in \cite{ehrhardt2023proximal}. Moreover, analogously to \Cref{thm:conv_estimate_general} we can as well deduce convergence in the Kullback-Leibler divergence for the running average over the iterates.
\begin{theorem}\label{thm:conv_estimate_strong}
Let \Cref{ass:strong} hold and select $(\lambda_k)_k$ and $(\tau_k)_k$ such that for all $k$, $\tau_k\leq \frac{m}{2L_{\nabla F}^2-m^2}$ and $\frac{\lambda_{k+1}}{\tau_{k+2}}\left(1-\frac{m\tau_{k+2}}{2}\right)\leq \frac{\lambda_k}{\tau_{k+1}}$. Then
\begin{equation}
\begin{aligned}
\KL(\nu_n^N | \pi)\leq \Lambda_{N,N+n}^{-1} \frac{\lambda_{N+1}}{2\tau_{N+2}}(1-\frac{m \tau_{N+2}}{2}) W^2_2(\mu Q^{N}_{\tau} S^G_{\tau_{N+1}},\pi)\\
+\Lambda_{N,N+n}^{-1} \sum\limits_{k=N+1}^{N+n} \lambda_k\left\{ \tau_{k+1} L_{\nabla F} d + \frac{1}{2}\tau_{k+1} L_G^2\|K\|^2 \right\}.
\end{aligned}
\end{equation}
\end{theorem}
\begin{proof}
The proof is analogous to the one of \Cref{thm:conv_estimate_general}.
\end{proof}

\begin{corollary}
Let \Cref{ass:strong} hold and $N\in\N$ be a burn-in time and $\epsilon>0$ arbitrary. Choose $\lambda_k=1$ and $\tau_k=\tau_\epsilon>0$ fixed for all $k$ such that
\begin{equation}
\begin{aligned}
\tau_{\epsilon}< \min\left\{\frac{\epsilon}{2}\left(L_{\nabla F} d + \frac{1}{2} L_G^2\|K\|^2\right)^{-1},\frac{m}{2L_{\nabla F}^2-m^2}\right\}.
\end{aligned}
\end{equation}
With $n_\epsilon\in\N$ satisfying $n_\epsilon>\frac{1}{\epsilon\tau_{\epsilon}}(1 - \frac{m\tau_{\epsilon}}{2}) W^2_2(\mu Q^{N}_{\tau} S^G_{\tau_{\epsilon}},\pi)$, for any $n\geq n_\epsilon$ it holds $\KL(\nu_n^N | \pi)<\epsilon$.
\end{corollary}

\section{Gradient-subgradient Langevin Sampling}\label{sec:grad_subgrad}
Since under \Cref{ass:strong} $F$ is differentiable, instead of a proximal step, an explicit gradient step for $F$ can also be employed. This leads to the gradient-subgradient Langevin algorithm (Grad-sub) depicted in \Cref{algo:grad_subgrad} which exhibits the same convergence behavior as Prox-sub, see \Cref{thm:grad_w2,cor:grad_w2}.
\begin{algorithm}
\setstretch{1.15}
\caption{Gradient-subgradient Langevin Algorithm (Grad-sub)}\label{algo:grad_subgrad}
\textbf{Input:} Initialization $X_0$, number of iterations $K$, step sizes $(\tau_k)_{k=0}^K$, $\tau_k>0$. \\ \vspace*{-\baselineskip}
\begin{algorithmic}[1]
\FOR{$k=0,1,2,\dots,K-1$}
\STATE Pick $Y_{k+1} \in \partial G(KX_k)$
\STATE $X_{k+\frac{1}{2}} = X_k - \tau_{k} K^* Y_{k+1}$
\STATE $X_{k+1} = X_{k+\frac{1}{2}} - \tau_{k+1}\nabla F(X_{k+\frac{1}{2}}) + \sqrt{2\tau_{k+1}}B_{k+1}$
\ENDFOR
\end{algorithmic}
\textbf{Output:} $X_K$.
\end{algorithm}

The Markov transition kernel corresponding to the explicit gradient step reads as
\[\bar{S}^F_{\tau}(x,A) = \delta_{x - \tau \nabla F(x)}(A)\]
and we denote the resulting transition kernel resembling \Cref{algo:grad_subgrad} as $\bar R_{\tilde{\tau},\tau}=S^G_{\tilde{\tau}}\bar{S}^F_\tau T_\tau$. For the explicit step on $F$, \cite[Lemma 4]{durmus2019analysis} yields for all $\tau\leq L_{\nabla F}^{-1}$ and $m\geq 0$
\[2\tau \left\{\Ec_F(\mu \bar{S}^F_\tau) - \Ec_F(\nu)\right\}
\leq (1- m\tau) W^2_2(\mu,\nu) - W^2_2(\mu \bar  S_\tau^F, \nu).\]
Replacing the estimate within \Cref{lem:KL_ineq_strong} leads to the following result.
\begin{proposition}\label{prop:grad_sub_eq}
Let \Cref{ass:strong} be satisfied with $m\geq 0$ and $\tau\leq L_{\nabla F}^{-1}$. Then for any $\mu \in \Pc_2(\R^d)$
\begin{equation}
\begin{aligned}
2\tau\left\{ \Fc(\mu \bar R_{\tilde{\tau},\tau})-\Fc(\pi)\right\} \leq
(1 - m\tau) W^2_2(\mu S^G_{\tilde{\tau}},\pi) - W^2_2(\mu \bar R_{\tilde{\tau},\tau} S^G_{\tau},\pi)\\ + \tau^2 L_G^2\|K\|^2 + 2 L_{\nabla F} d\tau^2.
\end{aligned}
\end{equation}
\end{proposition}
Thus, denoting  $\bar Q^k_{\tau} = \bar R_{\tau_1,\tau_2}\bar R_{\tau_2,\tau_3}\dots \bar R_{\tau_k, \tau_{k+1}}$ similar as before, \Cref{thm:conv_estimate_general} holds again and, the results in \Cref{sec:convergence_strong} hold with the modification that the factor $(1-\frac{m\tau}{2})$ is replaced by the slightly better $(1-m\tau)$.
\begin{theorem}\label{thm:grad_w2}
Let \Cref{ass:strong} be satisfied and select $(\tau_j)_j$ such that $\tau_j\leq L_{\nabla F}^{-1}$ for all $j$. Then for any $\mu \in \Pc_2(\R^d)$
\begin{equation}
\begin{aligned}
W^2_2(\mu \bar Q^k_{\tau} S^G_{\tau_{k+1}},\pi)\leq \prod\limits_{j=2}^{k+1}(1 - m\tau_j) W^2_2(\mu S^G_{\tau_1},\pi)  \\
+ \sum\limits_{j=2}^{k+1} \left(2 L_{\nabla F} d + L_G^2\|K\|^2\right) \tau_{j}^2 \prod\limits_{i=j+1}^{k+1}(1 - m\tau_i).
\end{aligned}
\end{equation}
In particular, for constant $\tau_k=\tau$ it holds true that
\begin{equation}
\begin{aligned}
W^2_2(\mu\bar  Q^k_{\tau} S^G_{\tau},\pi)\leq
(1 - m\tau)^k W^2_2(\mu S^G_{\tau},\pi)
+ \left(2 L_{\nabla F} d + L_G^2\|K\|^2\right)\frac{\tau}{m}
\end{aligned}
\end{equation}
\end{theorem}
\begin{corollary}\label{cor:grad_w2}
Let \Cref{ass:strong} hold  and pick $\tau_k=\tau_\epsilon$ for all $k$ and $n_\epsilon$ such that
\begin{equation}\label{eq:stepsize_grad}
\tau_\epsilon<\min\left\{\frac{m\epsilon}{2}\left(2 L_{\nabla F} d + L_G^2\|K\|^2\right)^{-1},L_{\nabla F}^{-1}\right\},\quad n_\epsilon>\frac{\log\left(\frac{\epsilon}{2W^2_2(\mu \bar S^G_{\tau},\pi)}\right)}{\log\left(1 - m\tau_\epsilon\right)}.
\end{equation}
Then for any $n\geq n_\epsilon$, it holds $W^2_2(\mu \bar Q^n_{\tau}  S^G_{\tau},\pi)<\epsilon$.
\end{corollary}

\begin{remark}
Note that \Cref{prop:grad_sub_eq} holds also for $m=0$ in which case we recover again, analogous results as for the Prox-sub algorithm.
\end{remark}

\section{Numerical Experiments}\label{sec:experiments}
In this section we provide numerical results obtained with the proposed methods in order to verify the proven convergence rates and to show the practical feasibility of the algorithms. The source code to reproduce the experiments of this paper is available at \cite{habring2023_git}. The experimental section is split into two parts. First we consider sampling from two-dimensional densities, which enables us to confirm the convergence rates in the space of measures with respect to the Wasserstein distance and the KL divergence, as these distances can be estimated well in the low-dimensional setting. Despite being computationally less challenging, the two-dimensional setting still allows us to incorporate a non-trivial linear operator $K$ so that we can cover all use-cases of interest with these experiments. Afterwards we consider experiments from the field of mathematical image processing, namely image denoising and deconvolution. Due to the high dimensionality of imaging problems, it is computationally not feasible to approximate Wasserstein distances or KL divergences between samples accurately, which is why we restrict ourselves to comparing estimated expected values based on the generated samples in these experiments.

\subsection{Example in $\R^2$}
We consider the functional $G(Kx)$ to be the total variation functional in which, in the 2D setting, corresponds to $G:\R\rightarrow\R$ and the operator $K:\R^2\rightarrow\R$ being defined as $G(p) = \lambda |p|$, $Kx = x_2-x_1$ with $\lambda>0$ a scaling parameter. One can easily check that $G$ is Lipschitz continuous with Lipschitz constant $L_G=\lambda$ and its subdifferential reads as $\partial G(p)=[-\lambda,\lambda]$ for $p=0$ and $\partial G (p) = \{ \text{sign}(p)\lambda\}$ else. Moreover, the operator $K$ satisfies  $\|K\|^2 \leq 2$. For the functional $F$ we will consider two choices in the following.

\subsubsection{$\TV-L^2$ Sampling}\label{sec:tv-l2-2d}
In order to provide a differentiable example lading to strong convergence rates as shown in \Cref{sec:convergence_strong} and \Cref{sec:grad_subgrad} we consider $F$ to be the squared $L^2$ discrepancy, i.e., $F:\R^2\rightarrow \R$ is defined as $F(x) = \frac{1}{2\sigma^2}\|x-y\|_2^2$. (Note that here we use $\|\cdot \|_2$ instead of just $\|\cdot \|$ to highlight the difference to the $1$-norm used below.) In this setting, $F$ has Lipschitz continuous gradient with Lipschitz constant $L_{\nabla F}=\sigma^{-2}$ and is $m$-strongly convex with parameter $m=\sigma^{-2}$. We choose $y=(-1,1)$, $\sigma=1$, and $\lambda=5$. We compare the results of Prox-sub, \Cref{algo:prox_subgrad}, and Grad-sub, \Cref{algo:grad_subgrad}. The proximal mapping of $F$ reads as ${\prox}_{\tau F}(x) = (1+\frac{\tau}{\sigma^2})^{-1}(x+\frac{\tau}{\sigma^2}y)$. Moreover, we find $\frac{m}{2L_{\nabla F}^2-m^2} = \frac{1}{L_{\nabla F}}=\sigma^2=1$ as the upper bound for the step size for both Grad-sub and Prox-sub.

We conducted experiments with the step sizes $\tau=1e-3, 1e-4, 1e-5$. In order to be able to approximate the distributions of iterates of the algorithms, we need i.i.d. samples of each iterate, thus, for each algorithm we compute 1e4 independent Markov chains and then estimate the distribution of the $k$-th iterate as a histogram based on these samples. Note that this increases the computational effort of the method and such a computation of parallel Markov chains could be avoided by establishing ergodicity of the proposed method. However, in the low-dimensional case vectorizing the computation of independent chains leads to a reasonably fast method and we refer the reader to the computation times listed in \Cref{table:computation_times_2d} for more details. To obtain a ground truth target to compare our results to we compute the normalization constant $Z = \int \exp(-F(x)-G(Kx))\;\d x$ approximately by numerical integration. Then we generate a piecewise constant approximation of the target density using the same bins as for the histogram of the samples. Afterwards we compute the Wasserstein distance between the two approximate distributions using Python's POT package \cite{flamary2021pot} and the KL divergence as well as the TV norm using their explicit formulas. (Note that the terminology \emph{total variation}, respectively TV, is used to describe two different things: Once, a norm on the space of (signed) measures allowing to quantify convergence and once a functional on the space of images.) This way we do not rely on any sampling procedure for estimating a ground truth. 

The convergence rates in Wasserstein distance are depicted in \Cref{fig:2d_strong_rates}. For all choices of $\tau$ the numerical convergence rates clearly exceed the expected ones. Moreover, the experimental results indicate that there is almost no difference between the convergence rates obtained with Prox-sub and Grad-sub despite the theory predicting slower convergence for Prox-sub. This is most likely due to the fact that, in the limit for vanishing step sizes, the explicit gradient step and the proximal step, which is in fact an implicit gradient step, are equivalent.%

\begin{figure}
\centering
\includegraphics[scale = 0.2]{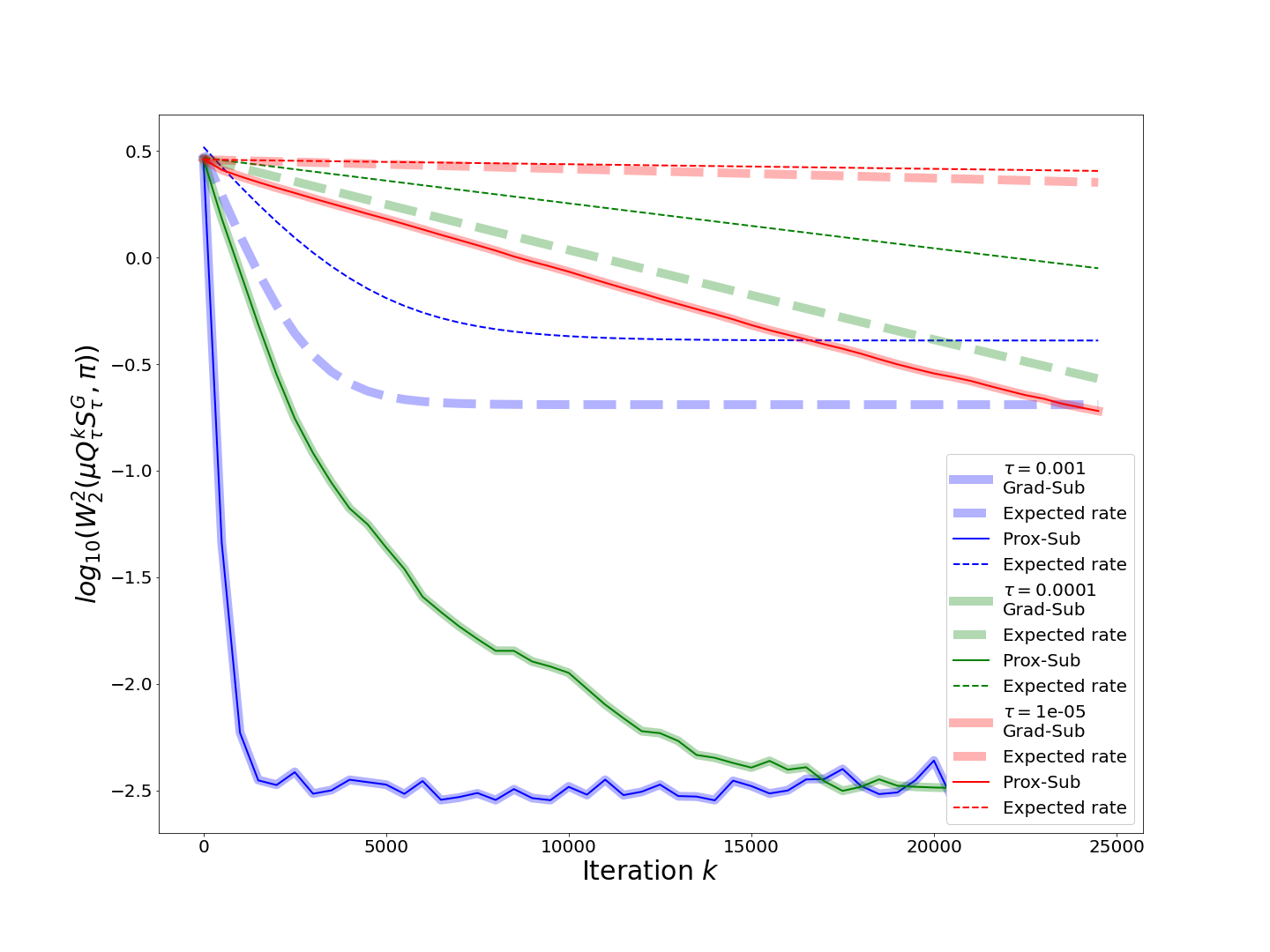}
\caption{Comparison of Grad-sub and Prox-sub for $\TV-L^2$ sampling in Wasserstein distance with $\sigma=1$, $\lambda=5$, $\tau=1e-3, 1e-4, 1e-5$. The color indicates the step size $\tau$, dashed lines the convergence rates obtained in the theory and solid lines the empirical ones, thick transparent lines the results from Grad-sub and thin lines the ones from Prox-sub.}
\label{fig:2d_strong_rates}
\end{figure}

Additionally, we also verify the convergence proven in \Cref{thm:conv_estimate_general} for the average distributions $\nu_k^N$ as defined in \eqref{eq:running_average}. This convergence holds identically for Grad-sub. We use no burn-in phase, that is $N=0$. The results are shown in \Cref{fig:2d_strong_rates_KL}. Again, we note that the proven convergence rates are clearly exceeded with Prox-sub and Grad-sub performing almost identically.
\begin{figure}
\centering
\includegraphics[scale = 0.2]{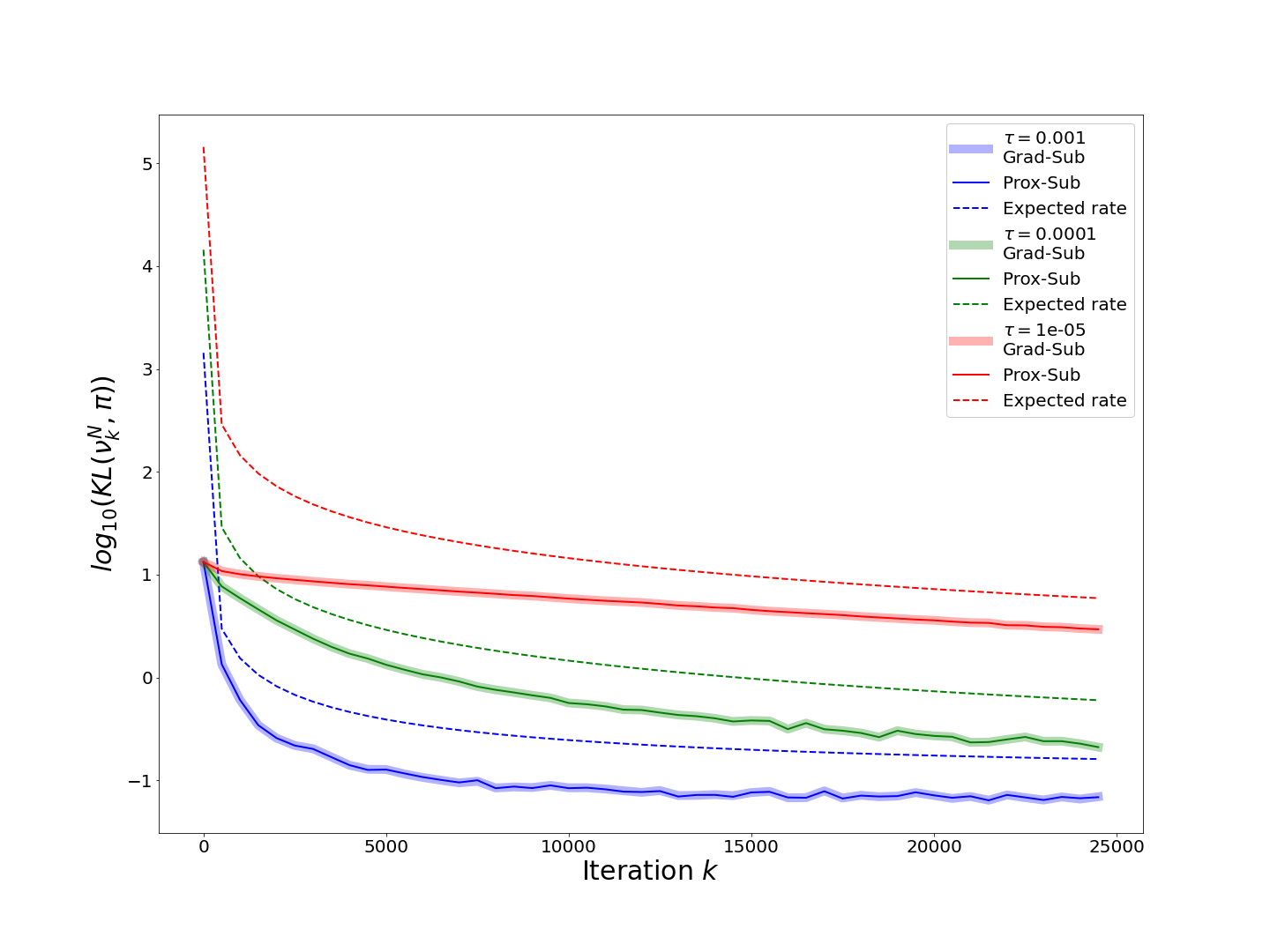}
\caption{Comparison of Grad-sub and Prox-sub for $\TV-L^2$ sampling in KL-divergence with $\sigma=1$, $\lambda=5$, $\tau=1e-3, 1e-4, 1e-5$. The color indicates the step size $\tau$, thick transparent lines the results from Grad-sub and thin lines the ones from Prox-sub. Dashed lines show the convergence rates obtained in the theory which are identical for both methods in this case.}
\label{fig:2d_strong_rates_KL}
\end{figure}

In \Cref{fig:rates_MALA_MYULA_TV_L2} we compare the results of Prox-sub to the ones obtained with P-MALA \cite{pereyra2016proximal} and MYULA \cite{durmus2022proximal}. For P-MALA a proposal for the update is computed as
\begin{equation}\label{eq:MALA}
\tag{P-MALA}
X^* = {\prox}_{\tau(F + G\circ K)}(X_k) + \sqrt{2\tau}B_{k+1}.
\end{equation}
Using a Metropolis-Hastings accept/reject step, this update is accepted for $X_{k+1}$ with probability $\min\left\{1,\frac{\pi(X^*)q(X_k|X^*)}{\pi(X_k)q(X^*|X_k)}\right\}$ and otherwise rejected leading to $X_{k+1}=X_k$. Here, $q(x|y) = \mathcal{N}(x;{\prox}_{\tau(F + G\circ K)}(y), 2\tau I_d)$ with $\mathcal{N}(x;\mu,\Sigma)$ being the density of a multivariate Gaussian with mean $\mu$ and covariance matrix $\Sigma$ evaluated at $x$ and $I_d\in\R^{d\times d}$ denoting the unit matrix. One iteration of MYULA, on the other hand, is defined as
\begin{equation}
\tag{MYULA}
X_{k+1} = (1-\frac{\tau}{\theta})X_k - \tau\nabla F(X_k) + \frac{\tau}{\theta}{\prox}_{\theta G\circ K}(X_k) + \sqrt{2\tau} B_{k+1}
\end{equation}
with a Moreau-Yoshida parameter $\theta>0$ and the step size requirement $\tau\leq\frac{\theta}{\theta L_{\nabla F}+1}$. In MYULA, the non-smooth part $G\circ K$ is replaced by its differentiable Moreau envelope
\begin{equation}
(G\circ K)^\theta(x) = \min\limits_z \frac{1}{2\theta}\|z-x\|^2 + G(Kz)
\end{equation}
which approximates $G\circ K$ as $\theta\rightarrow 0$. Applying standard ULA to the potential $U^\theta = F + (G\circ K)^\theta$ leads to the above iteration which approximates the target density in total variation as $\theta\rightarrow 0$. For this experiment we set $\theta=0.01$ yielding favorable convergence results. In \Cref{fig:rates_MALA_MYULA_TV_L2} we find very similar convergence behavior for P-MALA and MYULA. For P-MALA and MYULA convergence is proven in the TV-norm in \cite{durmus2022proximal,pereyra2016proximal}. For Prox-sub convergence in TV is only guaranteed for the average distributions $(\nu_k^N)_k$ which converge slower. However, as we could prove convergence of the iterates $\mu Q^k_\tau S_\tau^G$ directly as well (albeit in Wasserstein distance) we also plot the TV distance of these iterates, which decreases at the same speed as for P-MALA and MYULA.
\begin{figure}
\centering
\includegraphics[scale=0.2]{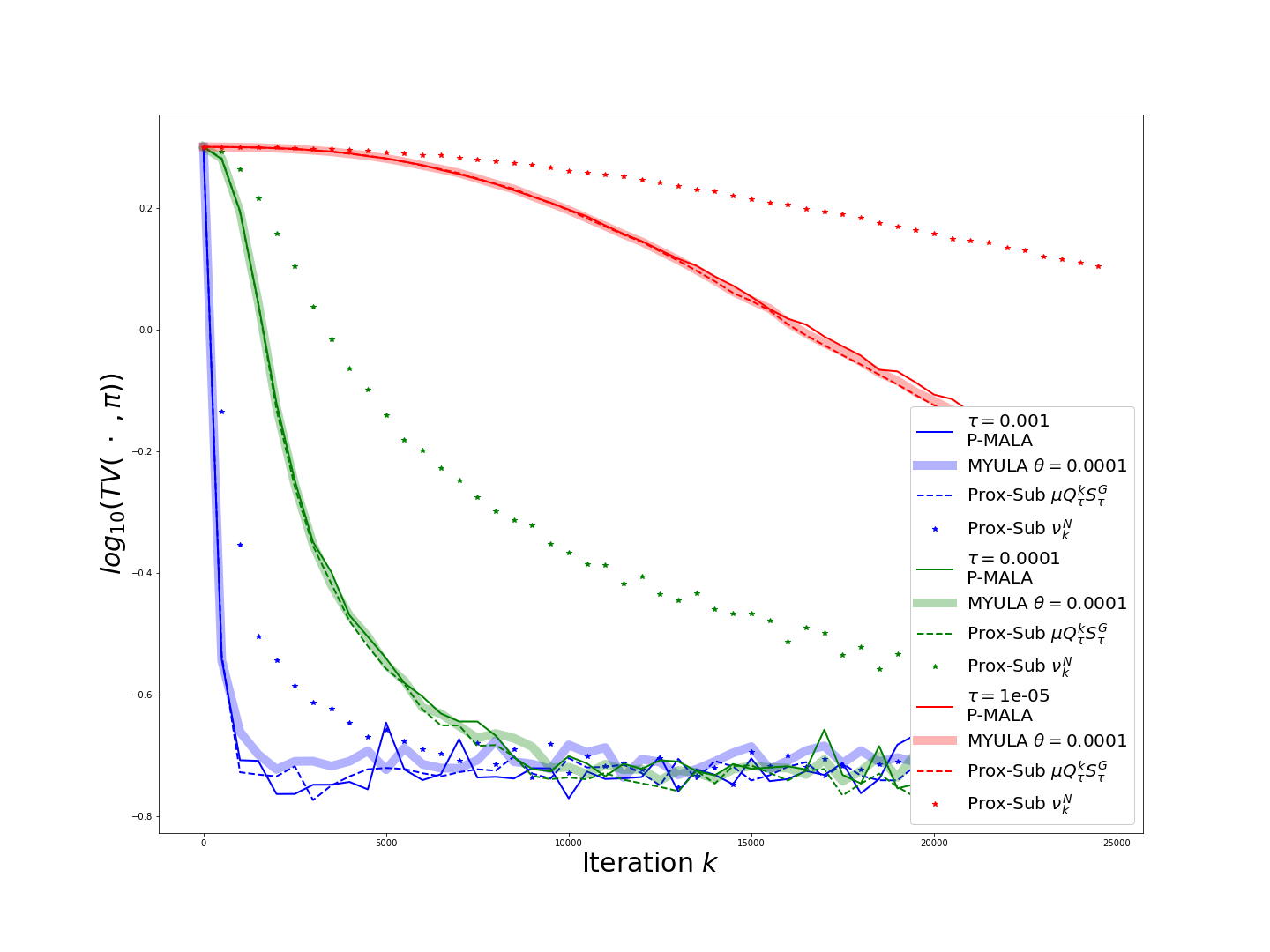}%
\caption{Comparison of Prox-sub, P-MALA \cite{pereyra2016proximal}, and MYULA \cite{durmus2022proximal} for $\TV-L^2$ sampling in TV distance with $\sigma=1$, $\lambda=5$, $\tau=1e-3, 1e-4, 1e-5$. For Prox-sub we show both, the total variation distance of the iterates $(\nu_k^N)_k$ for $N=0$ as well as for $\mu Q^k_\tau S_\tau^G$. For the former convergence is proven in TV whereas for the latter we could prove convergence in Wasserstein distance.}
\label{fig:rates_MALA_MYULA_TV_L2}
\end{figure}

For P-MALA and MYULA iterative computation of the proximal mappings is necessary. We do so by employing the the primal-dual algorithm from \cite{primal_dual_algo} which we terminate when the difference between consecutive iterates is less than $1e-4$ in the maximum norm. This leads to an increased computational burden for these two methods. The computation times are listed in \Cref{table:computation_times_2d} for different step sizes. Interestingly, we find that in particular the computation time of P-MALA is sensitive to the step size. This is due to the fact, that the step size constitutes the prox-parameter in \eqref{eq:MALA} with larger step sizes leading to more iterates within the primal-dual algorithm. For MYULA, on the other hand, the prox parameter is fixed as the hyperparameter $\theta$ which is why the computation times are less sensitive to changes in the step-size. In several cases the proposed methods are even faster computing 1e4 parallel chains than P-MALA and MYULA computing a single chain which means that the total computational burden with the proposed methods is lower despite accounting the lack of ergodicity by running parallel chains instead of using consecutive samples of a single chain for estimation.
\begin{table}[h]
\resizebox{\textwidth}{!}{%
\centering
\begin{tabular}{c|cccc|cccc}
& \multicolumn{4}{c}{1e4 chains} & \multicolumn{4}{c}{single chain}\\
$\tau$ & P-MALA & MYULA & Grad-sub& Prox-sub & P-MALA & MYULA & Grad-sub & Prox-sub\\
 \midrule
$10^{-5}$ & 2.11 & 22.83 & 0.5 & 0.52 & 0.36 & 1.02 & 0.019 & 0.24\\
$10^{-4}$ & 4.85 & 23.46 & 0.5 & 0.52 & 0.49 & 1.66 & 0.019 & 0.24\\
$10^{-3}$ & 35.59 & 23.61 & 0.50 & 0.52 & 1.93 & 1.77 & 0.019 & 0.24
\end{tabular}}
\vspace*{0.2cm}
\caption{Average computation times in seconds for 1000 iterations for different methods and step sizes $\tau$. We show the times for computing only a single Markov chain as well as for computation of 1e4 chains simultaneously in parallel.}
\label{table:computation_times_2d}
\end{table}

Due to $F$ not being Lipschitz continuous, convergence of the subgradient method \cite[Section 4.1]{durmus2019analysis} is not guaranteed. Nonetheless, since $F$ is locally Lipschitz and the iterates are typically bounded in practice, we apply the subgradient method (denoted by Sub in the sequel) as an additional comparison method. Note that for Sub, only convergence of the average distributions $(\nu^N_k)_k$ can be expected, which is why we evaluate convergence of those average distributions. For Prox-sub on the other hand, we show the errors for both the iterates $(\nu_k^N)_k$ as well as $(\mu Q^k_\tau S^G_\tau)_k$ itself, since we can ensure both to converge (albeit the latter in Wasserstein distance instead of KL). We employ the KL divergence between sampled and target distribution of Sub and Prox-sub as common metric to compare the convergence behavior as this is the metric where convergence is ensured for both methods, and also use a burn-in phase of $N=0$ for both methods. The results can be found in  \Cref{fig:2d_l2_subgradient}. Evidently, averaging over successive distributions within the computation of $\nu^N_k$ leads to a slower convergence.

\begin{figure}
\centering
\includegraphics[scale = 0.2]{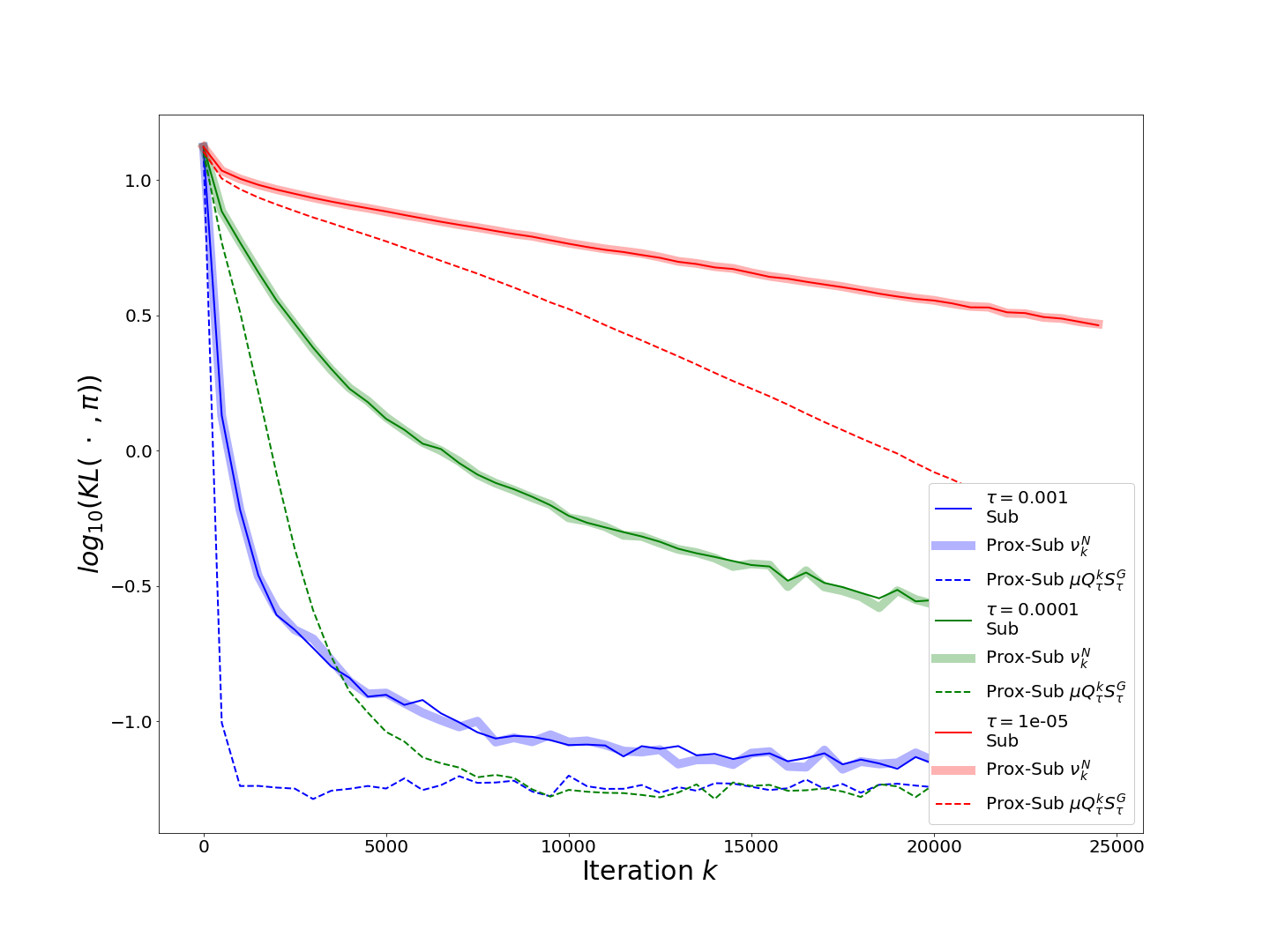}
\caption{Comparison of Prox-sub and Sub for $\TV-L^2$ sampling. For Prox-sub we show the KL divergence of both the iterates $\nu_k^N$ as well as $\mu Q^k_\tau S^G_\tau$, for Sub of the iterates $\nu^N_k$ where the burn-in phase is set to $N=0$.}
\label{fig:2d_l2_subgradient}
\end{figure}

\subsubsection{$\TV-L^1$ Sampling}
The proposed method Prox-sub is also applicable in the case that $F$ is non-differentiable. As an example for this case we consider $F$ to be the $L^1$ norm $F:\R^2\rightarrow \R$ given as $F(x) = \frac{1}{b}\|x-y\|_1$ with $b>0$. Such a model is in particular reasonable working with data corrupted with noise following a Laplace distribution. This choice for $F$ is Lipschitz continuous with Lipschitz constant $L_{F}=b^{-1}\sqrt{2}$. Moreover, the proximal mapping of $F$ in this case reads as ${\prox}_{\tau F} (x) = y + \text{sign}(x-y)\max\{0, |x-y|-\frac{\tau}{b}\}$. As before, we set $y=(-1,1)$, $b=1$, and $\lambda=5$. We compare the results of Prox-sub, and Sub, for which convergence can be ensured in this setting due to Lipschitz continuity of $F$. Results are shown again for the step sizes $\tau=1e-3, 1e-4, 1e-5$. As before an estimated ground truth comparison target is obtained by computing the normalization constant $Z = \int \exp(-F(x)-G(Kx))\;\d x$ by numerical integration and approximating the obtained density by a discretized distribution. We compute again 1e4 parallel Markov chains in this setting. The results are shown in \Cref{fig:2d_weak_rates_KL}. Once again, both methods satisfy the proven convergence rates with barely a noticeable difference in convergence speed in practical experiments. %
\begin{figure}
\centering
\includegraphics[scale = 0.2]{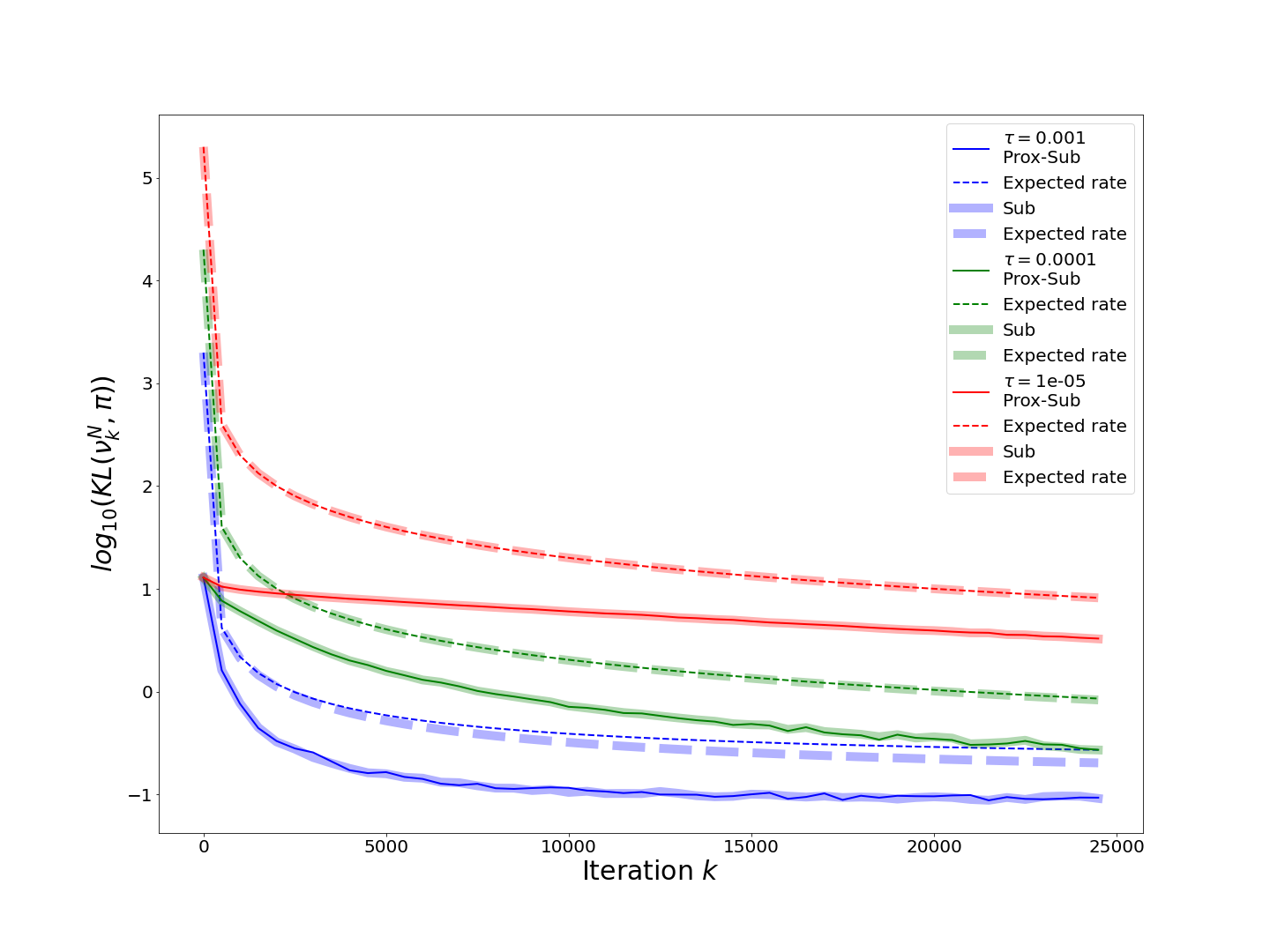}
\caption{Comparison of Prox-sub and Sub for $\TV-L^1$ sampling in KL-divergence with $b=1$, $\lambda=5$, $\tau=1e-3, 1e-4, 1e-5$. The color indicates the step size $\tau$, dashed lines the convergence rates obtained in the theory and solid lines the empirical ones, thick transparent lines the results from Sub and thin lines the ones from Prox-sub.}
\label{fig:2d_weak_rates_KL}
\end{figure}
We do not compare to P-MALA or MYULA as these methods are not particularly suited in this setting due to the necessary iterative computation of the proximal mapping of $F+G\circ K$ for non-differentiable $G$ and $F$. %

\subsection{Imaging Examples}
In this section we show experimental results for applications in mathematical imaging. Since computing Wasserstein distances or KL-divergences between image samples constitutes a significant computational burden due to the high dimensionality, we do not compare entire distributions, but restrict ourselves to comparing expected values within the following experiments. Note that convergence of the error in expected values is a consequence of convergence in Wasserstein 2-distance. Indeed, given an optimal coupling $(X,Y)$ for two distributions $(\mu,\nu)$, using Hölder's inequality, we find 
\[\|\E[X] - \E[Y]\|^2_2 = \|\E[X-Y]\|_2^2\leq \E[\|X-Y\|_2^2] = W^2_2(\mu,\nu).\]
In all imaging experiments, the functional $G(Kx)$ is chosen to be the anisotropic total variation, that is $G:\R^{2\times n\times m}\rightarrow\R$ and $K:\R^{n\times m}\rightarrow\R^{2\times n\times m}$ are defined as $G(p) = \lambda\|p\|_1 = \lambda\sum\limits_{i=1}^n\sum\limits_{j=1}^m|p_{i,j}^1| + |p_{i,j}^2|$ and 
\begin{equation}\label{eq:aniso_TV}
\begin{aligned}
(Kx)_{i,j}^1 = \begin{cases}
x_{i+1,j}-x_{i,j}\quad&\text{if $i<n$}\\
0\quad&\text{else,}
\end{cases}\quad
(Kx)_i^2 = \begin{cases}
x_{i,j+1}-x_{i,j}\quad&\text{if $j<m$}\\
0\quad&\text{else}.
\end{cases}
\end{aligned}
\end{equation}
We emphasize at this point that isotropic TV would be feasible as well. However, we use anisotropic TV since we want to compare our results to the ones obtained with Belief Propagation \cite{pea82,tapfre03,knobelreiter2020belief,szeliski2008comparative}, which is only applicable in the anisotropic case. The subdifferential of $G$ in this setting reads as
\begin{equation}
q\in \partial G(p)\Leftrightarrow \begin{cases} 
q_{i,j}^k = \lambda \quad&\text{if } p_{i,j}^k>0\\
q_{i,j}^k = -\lambda \quad&\text{if } p_{i,j}^k<0\\
q_{i,j}^k \in [-\lambda,\lambda] \quad&\text{if } p_{i,j}^k=0
\end{cases}
\end{equation}

\subsubsection{$\TV-L^2$ Denoising}\label{sec:numerical_im_denoising}
As a first imaging example we consider $\TV-L^2$ denoising. That is, $F:\R^{n\times m}\rightarrow\R$ is defined as $F(x) = \frac{1}{2\sigma^2}\|x-y\|_2^2$ for which the proximal mapping is analogous as in the two-dimensional case. As data $y\in\R^{n\times m}$ we use a ground truth image which is corrupted with additive Gaussian noise with mean zero and standard deviation $\sigma=0.05$. The upper bound for the step size according to our theoretical results is $\sigma^2 = 0.0025$ for both Grad-sub and Prox-sub. The scaling parameter $\lambda$ from \eqref{eq:aniso_TV} is chosen as $\lambda=30$. To obtain estimates of the ground truth for this setting we apply the Belief Propagation (BP) algorithm which yields highly accurate approximations of the marginal distributions of all image pixels, albeit for discrete gray-scale values. Specifically, we split the gray scale range of $[0,1]$ into $100$ evenly spaced values for the BP algorithm. Using the obtained marginal distributions we can compute the expected value as well as the pixel-wise marginal variances. For details on the application of BP to anisotropic $\TV-L^2$ denoising we refer the reader to \cite[Sections 4.3, 5.1]{narnhofer2022posterior} and to the publicly available source code \cite{habring2023_git}. To compute the expected values for the samples generated with the proposed methods we run the algorithm for a burn in phase of $N$ iterations and then compute the mean of samples $N+1,\dots,N+k$, which is denoted as $\bar{x}^N_k$. Thus, in the high-dimensional setting, we additionally decrease the computational effort by avoiding the computation of large amounts of independent Markov chains and instead using successive samples of a single chain for estimation. While a proof of ergodicity of the proposed algorithms is outside the scope of this paper and left for future work, we note that different versions of the Ergodic Theorem state that under some conditions the mean of successive iterates of a Markov chain approximates the expected value of the stationary distribution of this chain despite the samples not being i.i.d. \cite[Section 9.5]{grimmett2020probability}, \cite[Theorem 4.7.4]{robert1999monte}. Similarly to the mean, we use the same samples to compute the pixel-wise empirical variance. For all algorithms we use a burn-in phase of $N=1e6$ iterations and afterwards deploy the following $k=1e5$ samples.

In \Cref{fig:visual_image_denoising} we show results for BP, MYULA, Grad-sub and Prox-sub with a step size of $\tau = 1e-5$. The proximal mapping of $G\circ K$ for MYULA is again computed iteratively using the primal-dual algorithm \cite{primal_dual_algo} which we terminate as soon as the update within the algorithm is less than 1e-4 in the maximum norm. Moreover, for MYULA we choose a Moreau-Yoshida parameter of $\theta=1e-4$ leading a reasonable trade-off of speed and approximation accuracy. While smaller values lead to a better approximation of the target density, they also force us to use smaller step sizes due to the requirement $\tau\leq \frac{\theta}{\theta L_{\nabla F}+1}$ which leads to slower convergence. Moreover, smaller $\theta$ additionally increase the computational burden as they lead to a higher number of iterations within the primal-dual algorithm for computing the prox of $G\circ K$. We do not include numerical results for P-MALA as we could not achieve reasonable convergence. More specifically, in our experiments in order to achieve a reasonable acceptance rate of the Metropolis-Hastings step, the step size had to be chosen excessively small resulting in very slow convergence.

As expected we find high pixel-wise variances at the regions of jump discontinuities which is due to the very nature of the total variation functional penalizing sharp edges. Moreover, in \Cref{fig:convergence_image_denoising} we illustrate the squared $L^2$ error to the BP results for the proposed methods and MYULA for different step sizes. Again, we observe that there is hardly any difference in accuracy between Prox-sub and Grad-sub probably due to the small step sizes. As previously, we see that accuracy increases as the step size decreases. As the remaining bias suffers from the curse of dimensions, the reached accuracy is worse compared to the 2D examples. It seems that MYULA yields a greater bias than the proposed methods. Unfortunately, reducing this bias by decreasing $\theta$ would lead to smaller step sizes for MYULA and, thus, even slower convergence. Note that the accuracy of the reconstructed images compared to the clean images depends mostly on the model of the potential $U$ rather than the performance of the sampling algorithm. That is, a proficient sampling method should produce samples from $\pi(x) \propto \exp(-U(x))$ independent of how well this distribution is suited to acquire estimates of the unknown clean image. Therefore, we do not provide any results regarding the discrepancy of the MMSE estimates to the clean images as they do not provide any relevant information in the context of this research.
\newcommand{\h}{1.6cm}
\begin{figure}
\centering
\begin{subfigure}{\textwidth}
\centering
\includegraphics[height=\h]{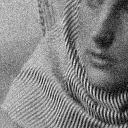}%
\includegraphics[height=\h]{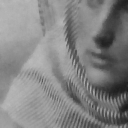}%
\includegraphics[height=\h]{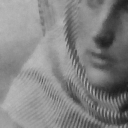}%
\includegraphics[height=\h]{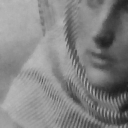}%
\includegraphics[height=\h]{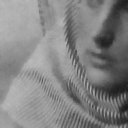}%
\includegraphics[height=\h]{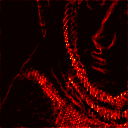}%
\includegraphics[height=\h]{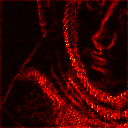}%
\includegraphics[height=\h]{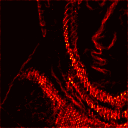}%
\includegraphics[height=\h]{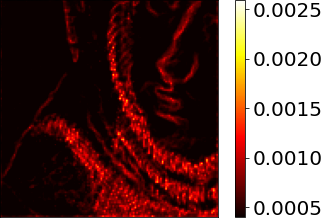}
\end{subfigure}
\begin{subfigure}{\textwidth}
\centering
\includegraphics[height=\h]{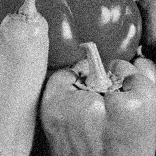}%
\includegraphics[height=\h]{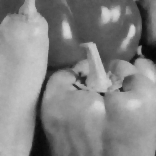}%
\includegraphics[height=\h]{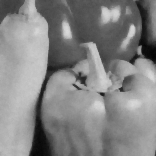}%
\includegraphics[height=\h]{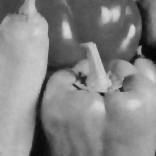}%
\includegraphics[height=\h]{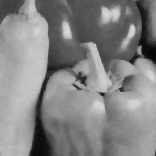}%
\includegraphics[height=\h]{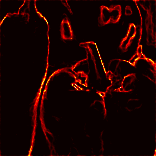}%
\includegraphics[height=\h]{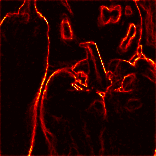}%
\includegraphics[height=\h]{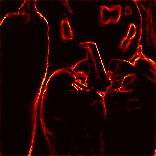}%
\includegraphics[height=\h]{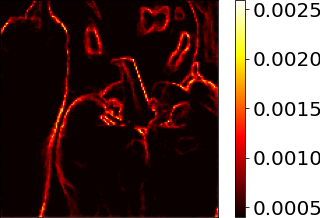}
\end{subfigure}
\caption{Expected values and marginal posterior variances for $\TV-L^2$ denoising computed with different methods. From left to right: Data $y$, expected value computed with BP, MYULA, Grad-sub, Prox-sub, marginal variances computed with BP, MYULA, Grad-sub, Prox-sub. The proposed algorithms were performed with a step size of $\tau=1e-5$, a burn-in phase of $N=1e6$ and the number of samples $k=1e5$. The range of the image pixels was $[0,1]$.}
\label{fig:visual_image_denoising}
\end{figure}
\begin{figure}
\centering
\includegraphics[scale=0.2]{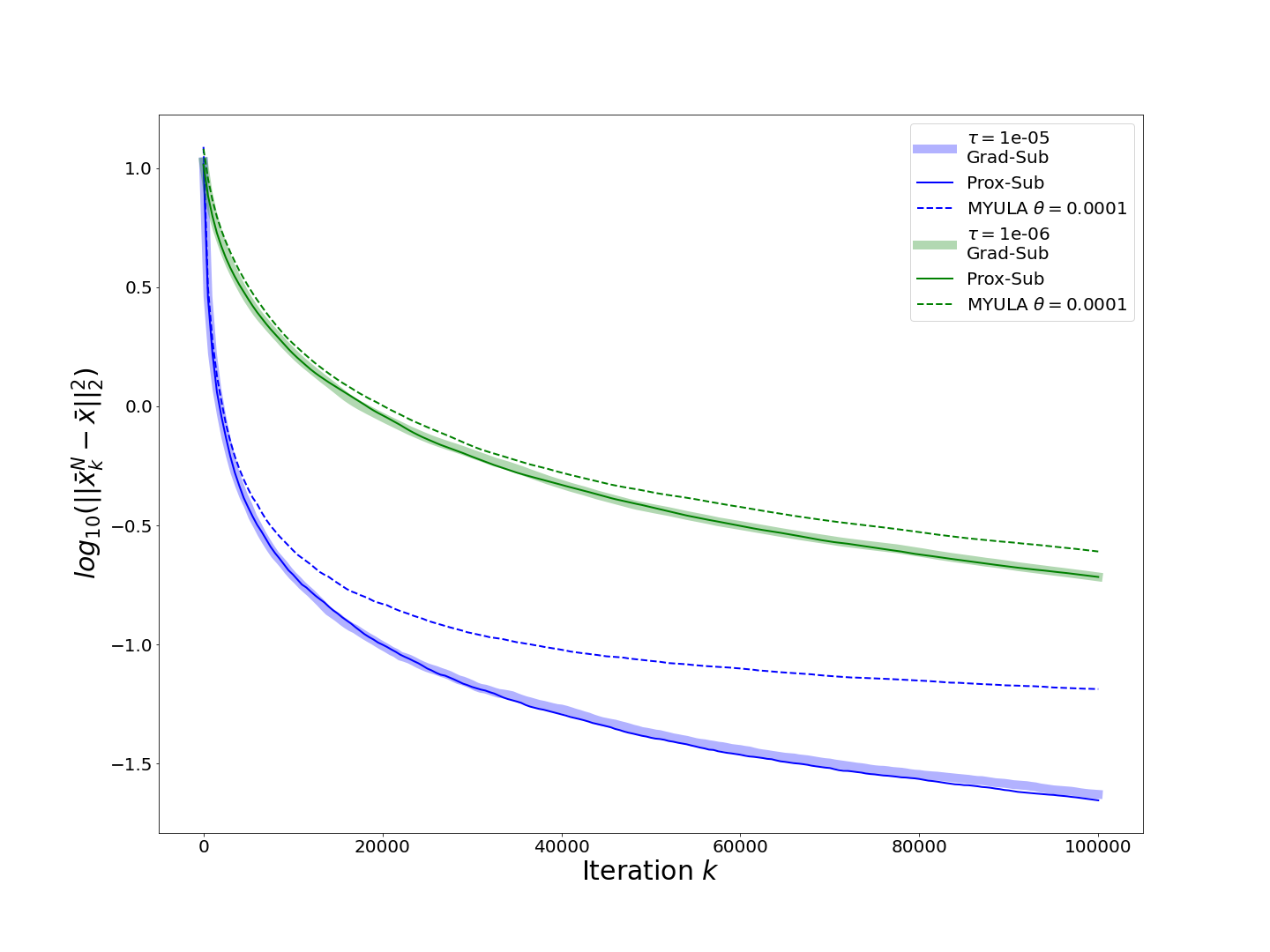}
\caption{Squared $L^2$ error between the estimated expected values and the ones obtained via BP for the peppers image for denoising.}
\label{fig:convergence_image_denoising}
\end{figure}

In \Cref{table:computation_times_image} we provide the average computation times for 1000 iterations of image denoising (and deconvolution) with Grad-sub and Prox-sub compared to MYULA, where we again find that the iterative computation of the proximal mapping within MYULA significantly increases computational effort compared to the proposed method. However, it should be mentioned at this point that the use of consecutive samples is heuristic for Grad-sub and Prox-sub as ergodicity is not yet proven for these methods.

\subsubsection{$\TV-L^2$ Deconvolution}
As a second example in the field of imaging we provide an application including a non-trivial forward operator within the functional $F$, namely image deconvolution. More precisely, $F:\R^{n\times m}\rightarrow\R$ is defined as $F(x) = \frac{1}{2\sigma^2}\|k*x-y\|_2^2$ where $k*x$ denotes the convolution of $x$ with a convolution kernel $k\in\R^{s\times s}$. While the application of Grad-sub does not constitute any difficulties as long as $F$ is differentiable, we require the computation of the proximal mapping of $F$ within Prox-sub. This is possible explicitly for this application since, in Fourier space, the convolution reduces to a pointwise multiplication with the Fourier transform of the convolution kernel, which renders the proximal mapping explicit. More precisely, using the facts that the Fourier transform is linear and unitary we find 
\[{\prox}_{\tau F}(x) = \Fc^{-1}\left((\mathbbm{1}+\frac{\tau}{\sigma^2}\hat{k}\odot\hat{k})^{\circ-1}\odot(\hat{x}+\frac{\tau}{\sigma^2}\hat{k}\odot \hat{y})\right)\]
where $\odot$ denotes element-wise multiplication (Hadamard product), $\mathbbm{1}$ a matrix with all entries set to one, $\hat{k}$ the discrete Fourier transform of $k$ and analogously for $x,y$, $A^{\circ-1}$ the element-wise inverse of a matrix $A$, and $\Fc^{-1}$ the inverse Fourier transform. Note that, the element-wise inverse is guaranteed to exist for $\tau$ small enough. As a convolution kernel $k$ we use a Gaussian kernel of size $s=5$ normalized to integrate to 1. The data y is obtained as the convolution $y = k*u + \sigma n$ with a ground truth image $u$ and $n$ denoting pixel-wise standard Gaussian noise. We choose a noise level $\sigma=0.01$. The scaling parameter $\lambda$ from \eqref{eq:aniso_TV} is chosen as $\lambda=20$. Since BP is not applicable in this setting, as a ground truth we apply Grad-sub with an added Metropolis-Hastings correction step, ensuring convergence to the stationary distribution \cite{roberts1996exponential}, \cite[Chapter 6]{robert1999monte}. Since it is not clear how long convergence to the stationary distribution takes using the Metropolis-Hastings correction, we use a significantly elongated burn-in phase of $N=3e6$ iterations and compute mean and pixel-wise variances using the subsequent $k=1e6$ samples to obtain our estimated ground truth comparison target. Again, we compare the proposed methods to MYULA for which we chose $\theta=1e-4$. Regarding the choice of $\theta$ we refer the reader to the elaborations in \Cref{sec:numerical_im_denoising}. For Prox-sub and Grad-sub we use a burn-in phase of $N=1e6$ iterations and the subsequent $k=5e5$ samples. Visual results are depicted in \Cref{fig:visual_image_deconv} for a step-size of $\tau=1e-6$ and convergence of the estimated expectation is illustrated in \Cref{fig:convergence_image_deconv}, where $\bar{x}$ denotes the the expected value computed using the Metropolis-Hastings correction and $\bar{x}^N_k$ again the estimates based on the proposed methods. Keep in mind, however, that inaccuracies within the estimated ground truth obtained via the Metropolis-Hastings algorithm might additionally distort the accuracy.
\begin{figure}
\centering
\begin{subfigure}{\textwidth}
\centering
\includegraphics[height=\h]{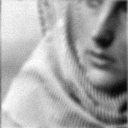}%
\includegraphics[height=\h]{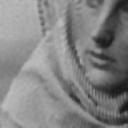}%
\includegraphics[height=\h]{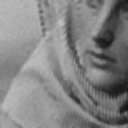}%
\includegraphics[height=\h]{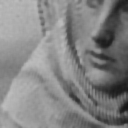}%
\includegraphics[height=\h]{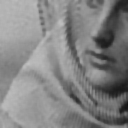}
\includegraphics[height=\h]{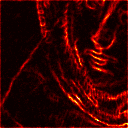}%
\includegraphics[height=\h]{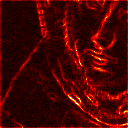}%
\includegraphics[height=\h]{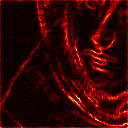}%
\includegraphics[height=\h]{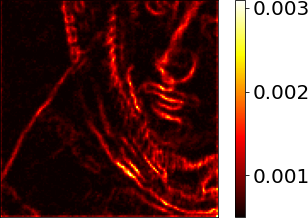}
\end{subfigure}
\begin{subfigure}{\textwidth}
\centering
\includegraphics[height=\h]{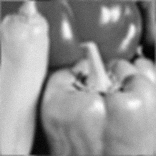}%
\includegraphics[height=\h]{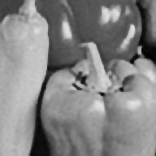}%
\includegraphics[height=\h]{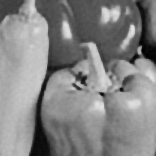}%
\includegraphics[height=\h]{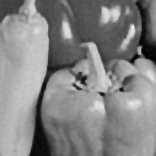}%
\includegraphics[height=\h]{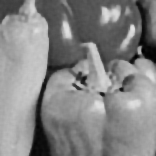}
\includegraphics[height=\h]{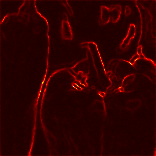}%
\includegraphics[height=\h]{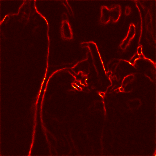}%
\includegraphics[height=\h]{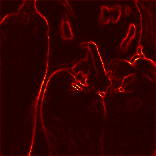}%
\includegraphics[height=\h]{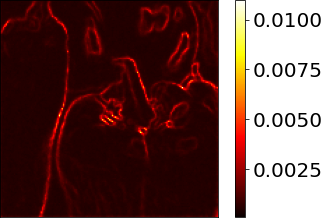}
\end{subfigure}
\caption{Expected values and marginal posterior variances for $\TV-L^2$ deconvolution computed with different methods. From left to right: Data $y$, expected value computed with Metropolis-Hastings correction, MYULA, Grad-sub, Prox-sub, marginal variances computed with Metropolis-Hastings correction, MYULA, Grad-sub, Prox-sub. The algorithms were performed with a step size of $\tau=1e-6$, a burn-in phase of $N=1e6$ and the number of samples $k=5e5$, the Metropolis-Hastings version with a burn-in phase of $N=3e6$ and the number of samples $k=1e6$. The range of the image pixels was $[0,1]$.}
\label{fig:visual_image_deconv}
\end{figure}

\begin{figure}
\centering
\includegraphics[scale=0.2]{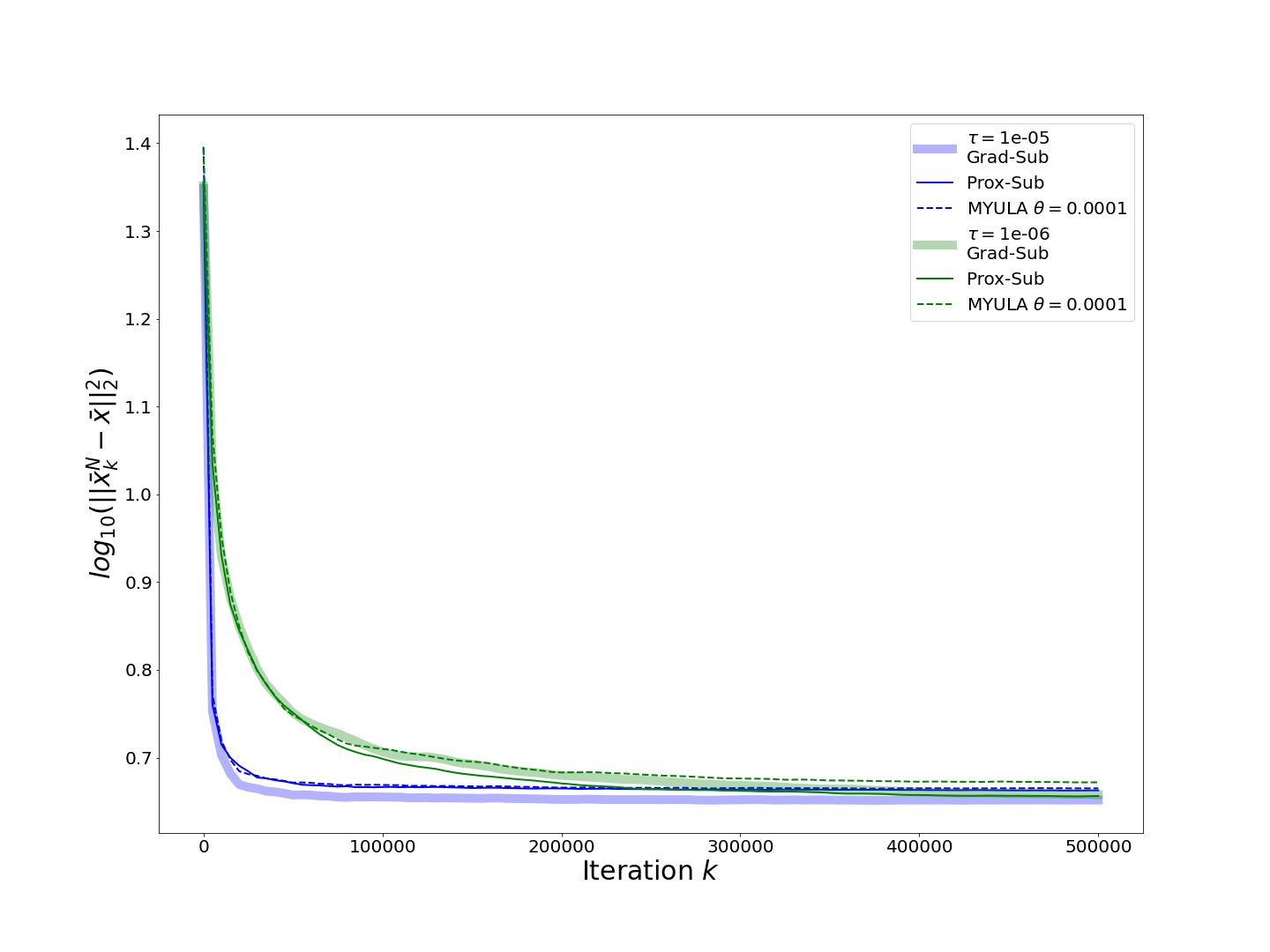}
\caption{Squared $L^2$ error between the estimated expected values using the proposed algorithms Grad-sub and Prox-sub and the ones obtained via the Metroplois-Hastings algorithm for the peppers image for deconvolution.}
\label{fig:convergence_image_deconv}
\end{figure}
In \Cref{table:computation_times_image} the computation times for image deconvolution for MYULA, Grad-sub, and Prox-sub are provided showing a similar result as above in that MYULA is computationally more expensive due to the iterative computation of proximal mappings.

\begin{table}[h]
\centering
\begin{tabular}{c|ccc|ccc}
& \multicolumn{3}{c|}{denoising} & \multicolumn{3}{c}{deconvolution}\\
$\tau$ & MYULA & Grad-sub & Prox-sub & MYULA & Grad-sub & Prox-sub\\
 \midrule
$1e-5$ & 55.61 & 0.59 & 0.65 & 43.43 & 0.92 & 1.06\\
$1e-6$ & 56.86 & 0.59 & 0.66 & 43.44 & 0.92 & 1.05
\end{tabular}
\vspace*{0.2cm}
\caption{Average computation times in seconds for 1000 iterations for different methods and step sizes $\tau$ for image denoising.}
\label{table:computation_times_image}
\end{table}

\section{Conclusion and Future Work}\label{sec:conclusion}
In this work we have proposed novel sampling algorithms for sampling from non-smooth densities of the form $\pi(x)\sim \exp(-F(x)-G(Kx))$ with functionals $F,G$ and a linear operator $K$. The methods are in particular applicable in the setting that $G$ is non-differentiable, and, for the first time, allow sampling with guaranteed convergence in applications with $F$ not being Lipschitz continuous. We have proven several non-asymptotic convergence results where convergence is obtained in the Kullback-Leibler divergence if $F$ is non-differentiable and in Wasserstein 2-distance with a faster rate in the case of stronger regularity assumptions on $F$. The presented experimental results confirm the proven convergence rates and provide empirical evidence of the practical relevance of the proposed methods in the field of Bayesian imaging.
\paragraph{Future Work}
An important issue to address in future work is existence and uniqueness of stationary measures for the proposed transition kernels $R_{\tilde{\tau},\tau}$ and, consequently, ergodicity of the resulting Markov chain which constitutes a difficult question due to the non-differentiability of the target density $\pi$. However, such results could reduce the computational complexity as it would obviate the need for running multiple Markov chains for the estimation of statistics of the target density $\pi$. This is due to the fact that ergodic theorems \cite[Chapter 17]{meyn2012markov} allow us to work with the running mean of a single chain instead of the mean over independent chains.

Another interesting line of possible future work concerns convergence guarantees of a primal-dual sampling algorithm inspired by \cite{chambolle2016ergodic}, as was for instance used in \cite{narnhofer2022posterior} without proving convergence. The possibility of avoiding step size constraints, similar to primal-dual minimization techniques, could provide a significant increase in speed of convergence.

\bibliographystyle{siamplain}
\bibliography{references}
\newpage
\appendix
\section{Proofs}\label{appendix_proofs}
\printProofs

\end{document}

%% file: shared.tex
\usepackage{lipsum}
\usepackage{amsfonts}
\usepackage{graphicx}
\graphicspath{ {images} }
\usepackage{epstopdf}
\usepackage{algorithmic}
\usepackage{mathtools}
\usepackage{bbm}
\usepackage{booktabs}
\usepackage{multirow}
\usepackage{makecell}
\usepackage{setspace}
\usepackage[createShortEnv]{proof-at-the-end}
\usepackage[shortlabels]{enumitem}
\ifpdf
  \DeclareGraphicsExtensions{.eps,.pdf,.png,.jpg}
\else
  \DeclareGraphicsExtensions{.eps}
\fi

\usepackage{subcaption}
\usepackage{tikz}
\usetikzlibrary{intersections}
\usetikzlibrary{angles, quotes}
\newcommand{\R}{\mathbb{R}}

\newcommand{\N}{\mathbb{N}}

\renewcommand{\d}{\text{d}}

\newcommand{\Bc}{\mathcal{B}}

\newcommand{\Ec}{\mathcal{E}}
\newcommand{\Hc}{\mathcal{H}}
\newcommand{\Pc}{\mathcal{P}}
\newcommand{\Fc}{\mathcal{F}}

\DeclareMathOperator*{\argmin}{argmin}

\DeclareMathOperator*{\prox}{prox}

\DeclareMathOperator*{\dom}{dom}

\DeclareMathOperator*{\TV}{TV}

\DeclareMathOperator*{\KL}{KL}
\DeclareMathOperator*{\Leb}{Leb}

\DeclareMathOperator{\p}{p}
\DeclareMathOperator{\E}{E}

\usepackage{enumitem}
\setlist[enumerate]{leftmargin=.5in}
\setlist[itemize]{leftmargin=.5in}

\newsiamremark{remark}{Remark}
\newsiamremark{example}{Example}
\newsiamremark{hypothesis}{Hypothesis}
\crefname{hypothesis}{Hypothesis}{Hypotheses}

\newtheorem{ass}{Assumption}

\crefname{example}{Example}{Example}
\newsiamthm{claim}{Claim}

\headers{Subgradient Langevin Sampling}{A. Habring, M. Holler, T. Pock}

\title{Subgradient Langevin Methods for Sampling from Non-smooth Potentials\thanks{Submitted to the editors 02.08.2023.}}

\author{Andreas Habring\thanks{Department of Mathematics and Scientific Computing, University of Graz, Austria (\email{andreas.habring@uni-graz.at}).}
\and Martin Holler\thanks{Department of Mathematics and Scientific Computing, University of Graz, Austria (\email{martin.holler@uni-graz.at}).}
\and Thomas Pock\thanks{Institute of Computer Graphics and Vision, Graz University of
Technology, Austria (\email{pock@icg.tugraz.at}).}}

\usepackage{amsopn}